\documentclass[11pt]{article}

\usepackage{amsmath,amsfonts,amsthm,amsrefs}
\usepackage{graphicx,color,pict2e}

\definecolor{refkey}{rgb}{1, 0.5, 0}  
\definecolor{labelkey}{rgb}{1,1,1}

\newtheorem{theorem}{Theorem}[section]
\newtheorem{lemma}[theorem]{Lemma}

\newtheorem{definition}[theorem]{Definition}

\def\defeq{\;\dot=\;}

\def\bel{\begin{equation}\label}
\def\eeq{\end{equation}}
\def\bega{\begin{array}}
\def\enda{\end{array}}

\title{Stationary Wave Profiles for Nonlocal Particle Models of Traffic Flow on Rough Roads}
\author{Jereme Chien and Wen Shen\footnote{ Mathematics Department, Pennsylvania State University, University Park, PA 16802,  U.S.A.  Emails: jpc6021@psu.edu and wxs27@psu.edu.}}

\begin{document}

\maketitle

\begin{abstract}
We study a nonlocal particle model describing traffic flow on rough roads. 
In the model, each driver adjusts the speed of the car according to the condition
over an interval in the front, leading to a system of nonlocal ODEs which we 
 refer to  as the FtLs (follow-the-leaders) model. 
Assuming that the road condition is discontinuous, 
we seek stationary wave profiles (see Definition~\ref{def:TWP}) 
for the system of ODEs across this discontinuity. 
We derive a nonlocal delay differential equation with discontinuous coefficient,
satisfied by the profiles, together with conditions on the asymptotic values as $x\to\pm\infty$.
Results on existence, uniqueness, and local stability are established for all cases.
We show that, depending on the case, there might exist a unique profile, 
infinitely many profiles, or no profiles at all. 
The stability result also depends on cases. 
Various numerical simulations are presented. 
Finally, we establish convergence of these profiles to those of a local particle model, 
as well as those of a nonlocal PDE model.
\end{abstract}


\section{Introduction and derivation of the model}
\setcounter{equation}{0}

Mathematical models for traffic flow has always been an active research field.
Commonly used models include the microscopic particle models and 
the macroscopic PDE models. 
See the classical book~\cite{Whitham} for a variety of models. 
Among the particle models, many adopt  the popular Follow-the-Leader (FtL) principle, 
where the behavior of a car depends on the leader/leaders ahead. 
Such models usually give rise to a large system of ODEs.
In the macroscopic models one treats the density of 
car distribution as the main unknown, and obtains nonlinear 
conservation laws where the total mass of cars is conserved.
The micro-macro convergence is of fundamental interests,
and results on first order models can be found in~\cite{MR3356989, HoldenRisebro, HoldenRisebro2, MR3605557, MR3217759, MR3541527, DFFRR2017},
and a recent work~\cite{FS2019} for space dependent flux.

In recent year, increasing research interests have been focused on nonlocal traffic models
and other related nonlocal models.
At the microscopic level, these models
describe that  the behavior of each car depends on the traffic pattern
over an interval of road of certain length. Such models are 
commonly referred to as the Follow-the-Leaders (FtLs) models.
Formally, the corresponding macroscopic models are typically 
conservation laws with nonlocal flux functions. 
When the road condition is uniform, 
various results are achieved on many aspects of the topic. 
Well posedness of the nonlocal conservation laws was obtained in~\cite{BG2016}
with a Lax-Friedrich type numerical approximation, 
and in~\cite{FKG2018} using a Godunov type scheme. 
The micro-macro limit, for suitable road condition, is treated~\cite{DFFR2019}. 
For nonlocal crowd dynamics and pedestrian flow,
 see~\cite{AggarwalGoatin2016, ColomboLecureuxMercier2011, ColomboGaravelloLecureuxMercier2011, ColomboGaravelloLecureuxMercier2012}.
On other models that result in nonlocal  conservation laws, we refer 
to~\cite{ChenChristoforou2007, ColomboMarcelliniRossi2016, DuKammLehoucqParks2012}.
See also results for several space dimensions~\cite{AggarwalColomboGoatin2015},
and other related works in~\cite{CrippaLecureuxMercier2013, Zumbrun1999}.

In this work we propose a nonlocal particle model for traffic flow with rough conditions
in one space dimension.
We now derive the model. 
We assume that all cars have the same length $\ell\in\mathbb{R}^+$,
and let  $z_i (t)$ be the position of the $i$th car at time $t$.  
We order the indices of  the cars such that 
\begin{equation}\label{zi}
z_i(t) \le z_{i+1}(t)-\ell \qquad \forall t\ge 0, \quad \mbox{for every}~i\in\mathbb{Z}.
\end{equation}
For a car with index $i$, we define the local discrete density perceived by the driver, 
depending on its relative position to its leader, 
\begin{equation}\label{defrho}
\rho_i(t)  \; \dot= \; \frac{\ell}{z_{i+1}(t)-z_i(t)}. 
\end{equation}
Note that if $\rho_i= 1$, then the two cars with indices $i$ and $i+1$ are
bumper-to-bumper. 
This physical constraint requires that 
$0 \le \rho_i(t) \le 1$ for all $i\in\mathbb{Z}$ and $t\ge 0$.
However, we remark that this constraint might fail for some models, 
see the discussion in section~\ref{sec:final} for a model that leads to crashing 
when $\rho_i$ becomes larger than 1. 

For a given $t\ge0$, based on the car distribution $\{z_i(t): i\in\mathbb{Z}\}$ 
and its corresponding discrete densities $\{\rho_i(t): i\in\mathbb{Z}\}$, 
we construct a piecewise constant density function as
\begin{equation}\label{eq:rhoell}
\rho^\ell(t,x) \;\dot=\; \rho_i(t) \qquad \mbox{for} ~x \in[z_i(t), z_{i+1}(t)) \quad \mbox{for} \quad 
i\in\mathbb{Z}, \qquad \forall t\ge 0.
\end{equation}

We let $w$ denote a  weight function, 
with support on the interval $x\in[0,h]$ for 
some given $h\in\mathbb{R^+}$.  
We assume that $w$ is bounded and
Lipschitz continuous  on its support and 
satisfies the assumptions
\begin{equation}\label{eq:w}
w(x) \ge 0 ~~ \forall x, \quad \int_0^h w(x)\, dx=1, \quad w(h)=0,\quad \mbox{and}\quad
w'(x)<0  ~\forall x\in[0,h]. 
\end{equation}
Although $w$ has bounded support, we define it on the whole real line. 
Note also  that the assumption in~\eqref{eq:w} implies that 
$w(x)$ is discontinuous at $x=0$, but continuous at $x=h$.

We assume that the road condition is varying, and let the speed limit $V(x)$
represent the condition at location $x$. 
More specifically, we consider rough road condition where $V$ is discontinuous,
and set
\begin{equation}\label{eq:V}
V(x) \defeq \begin{cases} V^- >0 , \qquad& \mbox{if}~x<0, \\ V^+ > 0, & \mbox{if}~x\ge0. \end{cases}
\end{equation}

Let $\phi: [0,1]\mapsto [0,1]$ be a $C^2$ function which  satisfies  the assumptions
\begin{equation}\label{eq:phi}
\phi(1)=0, \quad \phi(0)=1, \quad \mbox{and}
\quad  \phi'(\rho) < 0, \quad \phi''(\rho) \le 0~~ \mbox{for}~\rho\in[0,1].
\end{equation}

In this particle model, 
we assume that the speed of the car with index $i$ depends  on 
an average velocity $v^*$, where the average is taken over an interval 
of length $h$ in front of the car position $z_i$ with weight $w$. 
To be specific,  we let
\begin{equation}\label{FtLs}
 \dot z_i(t)  = v^*(z_i;\rho^\ell(t,\cdot)) , \qquad
 v^*(z_i;\rho^\ell(t,\cdot))\defeq 
  \int_{z_i}^{z_i+h} V(y)\phi(\rho^\ell(t,y)) w(y-z_i)\, dy.
 \end{equation}
The dot notation $\dot z_i$ denotes the time derivative of $z_i$. 
The dependence of $v^*$ on $\rho^\ell$ is nonlocal, through an integration.
Given an initial distribution of car positions $\{z_i(0)\}$, 
the system given by~\eqref{FtLs} 
indicates that the velocity of each car depends on a group of cars in front of it. 
We  refer to~\eqref{FtLs} as the {\em ``follow-the-leaders'' (FtLs) model}.  

This family of countably many ODEs~\eqref{FtLs} can be regarded as a nonlinear 
dynamical system on an infinite dimensional space. 
For example, one could set $y_i(t)=z_i(t)-z_i(0)$ and write the system~\eqref{FtLs}
as an evolution equation on the Banach space of bounded sequences of real numbers 
$y=(y_i)_{i\in\mathbb{Z}}$, with norm $\|y\|=\sup_i|y_i|$. 
For each $i\in\mathbb{Z}$, the right hand side of~\eqref{FtLs} is Lipschitz continuous,
even though $V(\cdot)$ is a discontinuous function.
For a given initial datum, the existence and uniqueness of solutions to this system
follow from the standard theory of evolution equations in Banach spaces,
see for example~\cite{Martin, SellYou}. 

The study on traveling wave profiles is fundamental for flow models. 
In this work we focus our attention on the stationary wave profiles 
for the FtLs model~\eqref{FtLs}, defined as follows.

\begin{definition}\label{def:TWP}
Let $\{z_i(t)\}$ be a solution of the FtLs model~\eqref{FtLs} with initial condition
$\{z_i(0)\}$. 
We say that $P(\cdot)$  is a \textbf{stationary wave  profile} for 
the FtLs model~\eqref{FtLs}  if 
\begin{equation}\label{P1}
P(z_i(t)) = \rho_i(t) =\frac{\ell}{z_{i+1}(t)-z_i(t)}  \qquad \forall i \in\mathbb{Z}, t\ge 0.
\end{equation}
\end{definition}

The equation satisfied by the profile $P$ will be derived in section~\ref{sec:2},
with suitable conditions on the asymptotic values $(\rho^-,\rho^+)$, 
see~\eqref{eq:dPx}-\eqref{eq:asymp}. 
We obtain a delay integro-differential equation with discontinuous coefficient. 
Such equations can be studied by an adapted version of the method of step~\cite{MR0477368,MR0141863}. 

Several recent works on analysis of traveling waves for related models are available.
For example,  when $h\to 0+$ and the weight function $w$ tends to a Dirac delta, and \eqref{FtLs} 
reduces to a local particle model. 
For this local particle model, in the simple case with uniform road condition 
where $V(x)\equiv 1$, the discrete 
traveling wave profiles are studied in~\cite{ShenKarim2017}. 
When $V$ is piecewise constant as in~\eqref{eq:V}, 
the discrete stationary wave profiles are treated in~\cite{ShenDDDE2017}. 
For the current nonlocal models~\eqref{FtLs} with $V(x)\equiv 1$, 
discrete traveling waves for particle models  as well as traveling waves for nonlocal 
conservation laws 
are analyzed in~\cite{RidderShen2018}.
Those traveling waves might be stationary or travel with a constant velocity. 
The existence, uniqueness (up to a horizontal shift), and local stability of the profiles
are established. 
The result for the corresponding nonlocal PDE models with piecewise constant $V$
can be found in~\cite{ShenTR}.

In this paper
we study the case where $V(x)$ is piecewise constant with a jump at $x=0$,
and we focus our attention on the
stationary wave profiles across the discontinuity in $V$.
We study existence, uniqueness, and local stability of the 
stationary wave profiles. 
Depending on the speed limits $(V^-,V^+)$ and asymptotic values
$(\rho^-,\rho^+)$, the results vary.
We show that, depending on the cases, 
(i) there exists a unique profile; (ii) there exist infinitely many profiles; 
or (iii) there exist no profiles at all. 
For the cases where the profiles exist, 
we show that some are time asymptotic solutions for the FtLs model~\eqref{FtLs},
while others are unstable and do not attract any nearby solutions of~\eqref{FtLs}.

\medskip

We also address the topic of various limits for the model~\eqref{FtLs}, 
in connection with traveling waves.
Formally, there are several limits of interests.

\medskip

\noindent\textbf{Limit 1: micro-macro limit to nonlocal conservation law.} 
Let $h$ be fixed and let $\ell \to 0$, 
 the solutions  $\rho^\ell$ of 
  the particle model~\eqref{FtLs} formally converges to  the solution of 
a nonlocal scalar conservation law
with discontinuous flux function
\begin{equation}\label{eq:clawNL}
\rho_t + \left[\rho \, \mathcal{A}(t,x;V,\rho)\right]_x =0, 
\qquad 
\mathcal{A}(t,x;V,\rho) \defeq \int_x^{x+h} V(y) \phi( \rho(t,y)) w(y-x)\; dy.
\end{equation}

\medskip

\noindent\textbf{Limit 2: nonlocal to local particle model.} 
Let $\ell$ be fixed and let $h\to 0$ and $w\to\delta_0$ (a dirac delta function), 
one obtains a local  particle model with 
\begin{equation}\label{eq:FtL}
 \dot z_i (t) = V(z_i) \phi(\rho_i(t)), \qquad \rho_i(t) = \frac{\ell}{z_{i+1}(t)-z_i(t)}
\qquad \forall i\in\mathbb{Z}, t\ge 0.
\end{equation}
In this model, the behavior of the car at $z_i$ depends only on a single  leader 
in front.  This is commonly referred to as the follow-the-leader (FtL) model. 

\medskip

\noindent\textbf{Limit 3: double-limit to local conservation law.} 
Furthermore, if we take the double limits $\ell \to 0$ and $h\to 0$, 
formally the model~\eqref{FtLs} converges to 
a local scalar conservation law with discontinuous flux
\begin{equation}\label{eq:claw}
\rho_t + \left[V(x) \rho \phi(\rho) \right]_x =0.
\end{equation}

\medskip 

Rigorous theoretical results on  the convergence of various limits  
are fundamental questions. 
Several recent results are of great interests, for the simpler case where
the road condition is uniform with $V(x)\equiv 1$. 
The micro-macro limit of a
nonlocal interaction equation with nonlinear mobility and symmetric kernel
is studied in~\cite{DFFR2019}.
For the PDE models, the nonlocal to local limit (i.e.~Limit 3) 
seems to be a complicated issue.
With downstream model, where the support of $w$ is ahead of the driver,
in~\cite{CCS2018} a counterexample is constructed where
the total variation of the density blows up instantly for certain initial data 
in BV.  
Furthermore, when the support of $w$ is around the driver, 
counter examples in~\cite{CCS2019} 
indicate that solutions for the model with nonlocal fluxes fail to converge to
those for the local  models.
See also~\cite{CCS2019P} for the effect of numerical viscosity in the study of this limit.
Convergence results (positive or negative) are still open for 
nonlocal follow-the-leaders models with weight defined on $[0,h]$ 
and $w'<0$. 

For the model~\eqref{FtLs} where the road condition are
discontinuous with $V$ in~\eqref{eq:V},  
and the assumptions on the parameters~\eqref{eq:w} and~\eqref{eq:phi}, 
these convergence results (micro-macro and nonlocal-local) 
are still open.
In this paper, as a first attempt in this direction,  
we establish the convergence for the stationary wave profiles  for Limits 1 and 2. 



The rest of the paper is organized as follows. 
In section 2 we derive the equation satisfied by the profiles, 
and we establish several technical Lemmas.
The case $V^->V^+$ is treated in section 3, and the case $V^-<V^+$ in section 4. 
Each case has 4 subcases, and various results on profiles are proved. 
In section 5 we prove the convergence of the profiles for Limit 1 and Limit 2.
Final concluding remarks are given in section 6, where we also discuss 
an alternative (possibly faulty) FtLs model on rough roads.

\section{Derivation of the profile equation and technical lemmas}
\label{sec:2} 

In this section we first derive the equation satisfied by the profile, 
then we present some technical lemmas and previous results
which will be useful in the analysis later.


\subsection{Derivation of the profile equation}
We  now derive  the equation satisfied by a stationary wave profile. 
Let $t\ge 0$ be given. 
Note that~\eqref{FtLs} can be rewritten as a system of ODEs for the
discrete density functions $\rho_i(\cdot)$, such that
\begin{align}\label{rhodot}
    \dot \rho_i(t) 
=-\frac{\ell \left( \dot z_{i+1} - \dot z_i \right) }{(z_{i+1}-z_i)^2} 
= 
\frac{1}{\ell} \rho_i^2(t) \cdot \big( v^*(z_i; \rho^\ell ) - v^*(z_{i+1}; \rho^\ell)
\big).
\end{align}
Here the piecewise constant function $\rho^\ell$ satisfies
\[
\rho^\ell (t,x) = \rho_i = P(z_i), \qquad \mbox{for}~ x\in[z_i,z_{i+1}), \quad \forall i\in\mathbb{Z}.
\]
Differentiating both sides of~\eqref{P1} in $t$ 
and using~\eqref{FtLs} and~\eqref{rhodot},
one gets
\begin{equation}\label{eq:dP}
P'(z_i) = \frac{\dot \rho_i}{\dot z_i} = 
\frac{P(z_i)^2}{\ell \cdot v^*(z_i;\rho^\ell(t,\cdot))} \Big[v^*(z_i;\rho^\ell(t,\cdot)) - 
v^*(z_{i+1};\rho^\ell(t,\cdot))\Big].
\end{equation}

We now introduce some notations. 
For a given profile $P$, we define an operator
\begin{equation} \label{eq:defL}
L^P(x) \;\dot=\; x + \frac{\ell}{P(x)},
\end{equation}
where $L^P(x)$ is the location of the leader for the car  at $x$. 
We remark that, in the particle model with the local density defined in~\eqref{defrho},
we have $\rho_i >0$ for all $i$.
The only exception is for the leader which could have an empty road in front 
and therefore  zero density, but
in our discussion on the stationary profile this never occurs. 
Thus we have  $P(x)>0$ and~\eqref{eq:defL} is well-defined. 

Furthermore, when the operator $L^P$ is composed with itself multiple times, 
we use the notation
\begin{equation} \label{eq:defLs}
(L^P)^k(x) \;\dot=\; \underbrace{L^P\circ L^P \cdots \circ L^P}_\text{k times}(x),
\qquad k\in\mathbb{Z}^+.
\end{equation}

Given a profile $P$, we define a piecewise constant function $P^\ell_{\{x\}}$ as
\begin{equation}\label{eq:Pell}
P^\ell_{\{x\}}(y)  \defeq P((L^P)^k(x)) 
\qquad \mbox{for} ~ y\in\left[(L^P)^k(x), \;(L^P)^{k+1}(x)\right), 
\quad \forall k \in\mathbb{Z}.
\end{equation}
With these notations, we now define
\begin{equation}\label{eq:v**}
v^*(x;P^\ell_{\{ \tilde x\}})  
\defeq \int_x^{x+h}  V(y)\cdot \phi(P^\ell_{\{\tilde  x\}}(y)) \cdot w(y-x)\; dy . 
\end{equation}
Note that $V$ and $P^\ell_{\{\tilde  x\}}$ are bounded piecewise constant functions 
so $x\mapsto v^*$ is Lipschitz continuous.
Since the  $z_i$ in \eqref{eq:dP} is arbitrarily chosen, we 
can replace it with $x$.  Using the notations introduced above, 
we can rewrite~\eqref{eq:dP} as a
\textit{delay integro-differential equation}
\begin{equation}\label{eq:dPx}
P'(x) = 
\frac{P^2(x)}{\ell \cdot v^*(x;P^\ell_{\{x\}})} \Big[v^*(x;P^\ell_{\{x\}})-
v^*(L^P(x);P^\ell_{\{x\}})\Big],
\end{equation}
where $v^*$ is defined in~\eqref{eq:v**}. 
We seek continuously differentiable 
solutions $P$ of~\eqref{eq:dPx} with the following asymptotic values 
\begin{equation}\label{eq:asymp}
\lim_{x\to -\infty} P(x) = \rho^-,\qquad \lim_{x\to +\infty} P(x) = \rho^+.
\end{equation}
Suitable conditions on $\rho^-,\rho^+$  will be specified later. 
We refer to~\eqref{eq:dPx}-\eqref{eq:asymp} as the \textbf{asymptotic value problem}. 

We remark that the delay in \eqref{eq:dPx} is 
strictly positive, of value larger than $\ell$. 
Fix a value $x_0 \in\mathbb{R}$.  If $P$ 
is given on the half line $x\ge x_0$, the equation \eqref{eq:dPx}
can be solved backward in $x$, as an ``initial value problem''. 
The existence and uniqueness of solutions to this 
initial value problem is a key step for the analysis of 
the asymptotic value problem for the stationary profiles.

\subsection{Technical Lemmas and previous results}

We now  present some technical lemmas and  previous results. 
We define the function
\begin{equation}\label{eq:f}
f(\rho) \;\dot=\;  \rho \phi(\rho), \qquad \rho\in[0,1].
\end{equation}
By the assumptions~\eqref{eq:phi},  we see that  
$f'' <0$ in the domain.
Furthermore, 
we let $\hat\rho$ be the unique value $\hat\rho$ such that 
\begin{equation}\label{eq:stag}
f'(\hat\rho ) =0.
\end{equation}

Next Lemma establishes the {\em ordering property} for the car distribution.

\begin{lemma}\label{lm:1}
Let $P$ be a continuously differentiable 
function that satisfies~\eqref{eq:dPx}, and
assume that  $0<P(x)<1$ for all $x\in\mathbb{R}$. Then 
\begin{equation} \label{eq:a1}
P'(x) < \frac{1}{\ell} P^2(x), \qquad \forall x  \in \mathbb{R}.
\end{equation}
Moreover, for every car at $x$, there exists a unique
follower $x^\flat$ such that $L^P(x^\flat)=x$.  
\end{lemma}

\begin{proof}
Since $0<P(x)<1$ for all $x$, we have that $0<\phi(P(x))<1$ and therefore 
$0<v^*(x;P^\ell_{\{x\}})$ for all $x$. 
Then  \eqref{eq:a1} follows immediately from  \eqref{eq:dPx}.

Next, fix an $x\in\mathbb{R}$ and let $x^\flat$ 
be the follower of $x$
such that  $L^P(x^\flat) =x$.
Then, it holds
\[
(L^P)'(x) = 1 - \frac{\ell}{P^2(x)}P'(x) >0 , \qquad \forall x\in\mathbb{R}.
\]
Thus the operator $L^P$ is monotone increasing, and therefore 
there exists a unique follower $x^\flat$. 
\end{proof}

We remark that 
the monotonicity of $x\mapsto L^P$  implies the ordering property, such that
$x<y$ if and only if $L^P(x) < L^P(y)$.

\begin{definition}\label{def:1}
Let $P$ be a continuously differentiable function such that 
\[
0<P(x)<1, \qquad  P'(x) < \frac{1}{\ell} P^2(x),\qquad \forall x\in\mathbb{R}.
\]
We call a sequence of car positions $\{z_i\}$ \textbf{a distribution generated by $P$}, if 
\[
 z_{i+1} = z_i + \frac{\ell}{P(z_i)}, \qquad \forall i\in\mathbb{Z}. 
\]
\end{definition}

Note that each given function $P$ can generate infinitely many distributions of $\{z_i\}$. 
However, if we fix the position of one car, say $z_0$, then the distribution is unique. 

Given a profile $P$ and a distribution $\{z_i\}$ generated by $P$ with 
$z_j=x$ for any fixed index $j$,  
the piecewise constant function $P^\ell_{\{x\}}$ defined in~\eqref{eq:Pell}
satisfies
\begin{equation}\label{eq:Pell2}
P^\ell_{\{x\}}(y) = P(z_i) \quad \mbox{for} ~ y\in[z_i,z_{i+1}) \quad \forall i \in\mathbb{Z},
\qquad z_j=x ~\mbox{for some $j$}.
\end{equation}
In the rest of the paper we denote by $P^\ell_{\{x\}}$ 
the piecewise constant function associated with $P$.

The next Lemma is immediate.

\begin{lemma}\label{lm:2}
Let $P$ be a stationary profile that satisfies \eqref{eq:dPx}, and let $\{z_i(0)\}$
be a distribution generated by $P$.  Let $\{z_i(t)\}$ be the solution of the 
FtLs model \eqref{FtLs} with initial condition $\{z_i(0)\}$. 
Then, $\{z_i(t)\}$ is a distribution generated by $P$ for all $t\ge 0$.
\end{lemma} 

In other words, let $\{\rho_i(t)\}$ be the corresponding discrete density for $\{z_i(t)\}$, 
then $P(z_i(t)) = \rho_i(t)$ for all $t\ge 0$ and for all $i\in\mathbb{Z}$. 


We now define the concept of periodic solutions for the FtLs model.

\begin{definition}\label{def:period}
Let $\{z_i(t): i\in\mathbb{Z}\}$ be the solution of~\eqref{FtLs} with initial condition $\{z_i(0): i\in\mathbb{Z}\}$.
We say that $\{z_i(t): i\in\mathbb{Z}\}$ is \textbf{periodic} if there exists a constant $t_p\in\mathbb{R}^+$,
independent of $i$ and $t$, such that
\begin{equation}\label{PT}
z_i (t+t_p) = z_{i+1}(t), \qquad \forall i\in\mathbb{Z}, ~ \forall t \ge 0.
\end{equation}
We call $t_p$ the \textbf{period}.
\end{definition}

Definition~\ref{def:period} indicates that, in a periodic solution $\{z_i(t): i\in\mathbb{Z}\}$, 
after a time period of $t_p$, each car takes over the position of its leader. 
In the next Lemma we show that this is closely related to stationary wave profiles.

\begin{lemma}\label{lm:3}
(i) Let $P$ be a continuously differentiable function and $0<P(x)<1$ for all $x\in\mathbb{R}$,
and let $P^\ell_{\{x\}}$ be the associated piecewise constant function.
Then, $P$ satisfies \eqref{eq:dPx} if and only if 
\begin{equation}\label{eq:a3}
\int_x^{L^P(x)} \frac{1}{v^*(z;P^\ell_{\{z\}})} dz = C ,  \qquad \forall x \in\mathbb{R},
\end{equation}
for some constant $C\in\mathbb{R}^+$. 

(ii) Moreover, let $\{z_i(t)\}$ be the solution of \eqref{FtLs} with initial data $\{z_i(0)\}$
which is a distribution generated by $P$. 
Then, $\{z_i(t)\}$ is periodic  if and only if $P$ satisfies \eqref{eq:dPx}.
The period $t_p$ equals the constant $C$ in \eqref{eq:a3}.
\end{lemma}

\begin{proof}
\textbf{(i)} We assume that \eqref{eq:a3} holds. Differentiating \eqref{eq:a3} in $x$ on both 
sides, we get
\begin{equation}\label{eq:a4}
(L^P)'(x) \frac{1}{v^*(L^P(x);P^\ell_{\{x\}})} - \frac{1}{v^*(x;P^\ell_{\{x\}})} =0 .
\end{equation}
Using $(L^P)'(x)= 1-\ell P'(x) / P^2(x)$ and $0<P(x)<1$, 
we easily deduce \eqref{eq:dPx}.

Now assume that \eqref{eq:dPx} holds, which implies \eqref{eq:a4}, and further implies \eqref{eq:a3}.  Therefore, \eqref{eq:a3} and \eqref{eq:dPx} are equivalent.

\textbf{(ii).} 
Assume that $P$ satisfies \eqref{eq:dPx}, and let $\{z_i(t)\}$ be the solution of 
\eqref{FtLs} with initial data $\{z_i(0)\}$ generated by $P$. 
Fix any time $t\ge 0$, and an index $i\in\mathbb{Z}$. 
Since $v^*(z_i;P^\ell_{\{z_i\}})$ does not depend on $t$ explicitly, 
\eqref{FtLs}  is separable. 
Let $t_{p,i}$ be the time it takes for a car at $z_i$ to reach it leader's position 
$z_{i+1}= z_i+\ell/P^\ell(z_i)$, we compute
\[
 t_{p,i}
= \int_t^{t+t_{p,i}} dt =
\int_{z_i}^{z_i + \ell/P^\ell(z_i)} \frac{1}{v^*(z_i;P^\ell_{\{z_i\}}) } d z_i .
\]
By \eqref{eq:a3} we conclude that $t_{p,i}=C=t_p$ is constant, therefore $\{z_i(t)\}$
is periodic.
The reverse implication is easily proved by reversing the above argument. 
\end{proof}


We are interested in the asymptotic value problem of~\eqref{eq:dPx} 
with the asymptotic conditions~\eqref{eq:asymp}.  
Defining $z^\flat$ as
\begin{equation}\label{eq:zflat}
L^P(z^\flat)  =-h,
\end{equation}
we observe that on $x\in(-\infty,z^\flat)\cup (0,\infty)$ 
the equation~\eqref{eq:dPx} for $P$ is the same as the 
one for the stationary profile for the FtLs model with $V(x)\equiv 1$,
studied in~\cite{RidderShen2018}.  
We denote by $W$ the profiles studied in~\cite{RidderShen2018}. 
We further recall that in~\cite{RidderShen2018}
the existence, uniqueness (up to a horizontal shift) and local
stability of the stationary profiles $W$ were established. 

Thanks to the follow-the-leaders principle, 
we conclude that  on $x\ge 0$, 
the profiles $P$ (if they exist) must match $W$. 
Furthermore, similar results are valid on asymptotic behaviors as $x\to\pm\infty$,
for $P$ and for $W$. 
We recall the following result from~\cite{RidderShen2018}*{Lemma 2.2}.

\begin{lemma}\label{lm:RS}
Let $\rho^-,\rho^+\in\mathbb{R}^+$ be given and assume $\rho^\pm \in (0,1)$.
Assume that $P$ is a solution of~\eqref{eq:dPx}, 
continuously differentiable on $x>0$ and 
$x<z^\flat$,
satisfying the asymptotic conditions~\eqref{eq:asymp}.
Then, the following holds.
\begin{itemize}
\item 
As $x\to+\infty$, $P(x)$ approaches $\rho^+$ with an exponential rate 
if and only if $\rho^+ > \hat \rho$. 
\item 
As $x\to-\infty$, $P(x)$ approaches $\rho^-$ with an exponential rate 
if and only if $\rho^-< \hat \rho$. 
\end{itemize}
\end{lemma}

Lemma~\ref{lm:RS} implies that, if $\rho^-$ and $\rho^+$ are 
stable asymptotic values
at $x\to -\infty$  and $x\to +\infty$ respectively, 
for monotone profiles, then we must have
$\rho^- < \hat \rho < \rho^+$.
In the sequel  we say that $\rho^-$ (and $\rho^+$ respectively) is
a \textbf{stable asymptote}  if $\rho^-<\hat\rho$ 
(and $\rho^+>\hat\rho$ respectively). 
On the other hand, we say that $\rho^-$ (and  $\rho^+$ respectively)  is
an \textbf{unstable asymptote}  if $\rho^->\hat\rho$ 
(and  $\rho^+<\hat\rho$ respectively). 
Furthermore, combined with the periodic behavior in Lemma~\ref{lm:3}, 
we immediately have the next Lemma,  which establishes
properties on the flux and density values at $x\to\pm\infty$.

\begin{lemma}\label{lm:ASF}
Assume that $P$ is a piecewise continuously differentiable function
that satisfies~\eqref{eq:dPx}, and let $\rho^-,\rho^+$ be 
asymptotic values 
such that~\eqref{eq:asymp} holds. 
Then there exists a value $\bar f \in\mathbb{R}^+$ such that 
\begin{equation} \label{eq:cc}
V^- \cdot f(\rho^-) = V^+ \cdot f(\rho^+) = \bar f.
\end{equation}
Let $\{z_i(0)\}$ be a distribution generated by this $P$, and 
let $\{z_i(t)\}$ be the solution of~\eqref{FtLs} with initial condition $\{z_i(0)\}$.
Then  $\{z_i(t)\}$ is periodic with period  $t_p = \ell/\bar f$, i.e.,
\begin{equation}\label{eq:pp}
t_p = \int_x^{L^P(x)} \frac{1}{v^*(z; P^\ell_{\{z\}}) } \; dz =\frac{\ell}{\bar f} , \qquad \forall x\in\mathbb{R}. 
\end{equation}
\end{lemma}

For~\eqref{eq:dPx}-\eqref{eq:asymp}
we will show later that, for some cases  there exist infinitely many stationary profiles.
If this happens, the next Lemma shows that these profiles never cross each other.

\begin{lemma}\label{lm:order}
Let $P_1$ and $P_2$ be two distinct profiles that satisfy~\eqref{eq:dPx}-\eqref{eq:asymp}. 
Then 
the graphs of $P_1$ and $P_2$ never cross each other. 
\end{lemma}

\begin{proof}
We prove by contradiction. 
Assume the opposite, that the graphs of $P_1$ and $P_2$ 
cross each other, and
let $y$ be the rightmost crossing point.
Without loss of generality we assume that
\begin{equation} \label{eq:ss}
P_1(y)=P_2(y), \qquad P_1(x) > P_2(x) \quad \forall x>y. 
\end{equation}
Let  $y^\sharp \doteq L^{P_1}(y) = L^{P_2}(y)$ 
be the position of the leader for the car at $y$, 
in both car distributions generated by $P_1$ and $P_2$,
and let $t_{p,1}$ and $t_{p,2}$ be the periods
\[
t_{p,1}  = \int_y^{y^\sharp} \frac{1}{v^*(x;P_{1,\{x\}}^\ell)}dx, \qquad 
t_{p,2}  = \int_y^{y^\sharp} \frac{1}{v^*(x;P_{2,\{x\}}^\ell)}dx.
\]
By our assumptions~\eqref{eq:ss} we have $t_{p,1}>t_{p,2}$. 
But since $P_1,P_2$ have the same asymptotic values, 
by Lemma~\ref{lm:3} and Lemma~\ref{lm:ASF} 
we have that $t_{p,1} =t_{p,2} $, a contradiction. 
\end{proof}

%

\section{Case 1: $V^->V^+$} 
\label{sec:case1}
\setcounter{equation}{0}
In this section we study the case $V^->V^+$, i.e., the speed limit 
has a downward jump at $x=0$. 
We prove results related to existence, uniqueness and stability of the
stationary profiles.
To simplify the notations, we introduce the functions
\begin{equation} \label{eq:}
f^-(\rho) \doteq V^- \rho \phi(\rho) = V^- f(\rho), \qquad 
f^+(\rho) \doteq V^+ \rho \phi(\rho) = V^+ f(\rho), \qquad  \rho\in[0,1].
\end{equation}

By Lemma~\ref{lm:ASF},  the asymptotic values $(\rho^-,\rho^+)$ 
must satisfy 
$ f^-(\rho^-)=f^+(\rho^+)=\bar f$, for some
value $\bar f\in\mathbb{R}^+ $ in the range of both functions $f^-$ and $f^+$. 
The horizontal line $f=\bar f$ intersects twice with each graph of  $f^-$ and $f^+$,
see Figure~\ref{fig:rhos}. We have
\begin{equation}\label{eq:rhoss}
0\le \rho_1 < \rho_2 \le \hat \rho \le \rho_3 < \rho_4 \le 1, \quad
f^-(\rho_1)=f^-(\rho_4)=f^+(\rho_2)=f^+(\rho_3)=\bar f.
\end{equation}
%
%

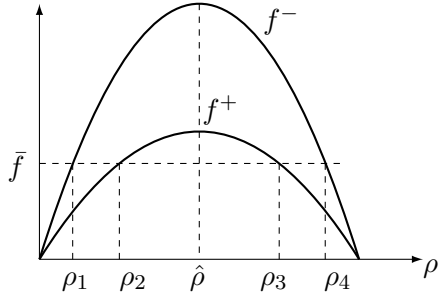
\begin{figure}[htbp]
\begin{center}
\setlength{\unitlength}{0.85mm}
\begin{picture}(60,50)(-3,-5)  
\put(0,0){\vector(1,0){60}}\put(60,-2){$\rho$}
\put(0,0){\vector(0,1){40}}
\multiput(0,15)(2,0){24}{\line(1,0){1}}\put(-5,13){$\bar f$}
\multiput(5.2,15)(0,-2){8}{\line(0,-1){1}}\put(2,-4){ $\rho_1$}
\multiput(12.5,15)(0,-2){8}{\line(0,-1){1}}\put(11,-4){ $\rho_2$}
\multiput(25,40)(0,-2){20}{\line(0,-1){1}}\put(22,-4){ $\hat\rho$}
\multiput(37.5,15)(0,-2){8}{\line(0,-1){1}}\put(33,-4){ $\rho_3$}
\multiput(44.7,15)(0,-2){8}{\line(0,-1){1}}\put(43,-4){ $\rho_4$}
\put(25.5,22){$f^+$}\put(35,36){$f^-$}
\thicklines
\qbezier(0,0)(25,40)(50,0)
\qbezier(0,0)(25,80)(50,0)
\end{picture}
\caption{Graphs of the functions $f^-, f^+$, 
and locations of $\rho_1,\rho_2,\rho_3,\rho_4$ and $\hat\rho$.}
\label{fig:rhos}
\end{center}
\end{figure}

There are 4 subcases in total: 

\begin{center}
\begin{tabular}{|c|c|c|c|c|}
\hline
subcase & 1A & 1B & 1C & 1D \\
\hline
$(\rho^-,\rho^+)$ & $(\rho_1,\rho_2)$ & $(\rho_1,\rho_3)$&$(\rho_4,\rho_3)$&$(\rho_4,\rho_2)$\\
\hline
\end{tabular}
\end{center}
%

We first observe that the cases with $\bar f =0$ are trivial. 
In this case we have $\rho_1=\rho_2=0$ and $\rho_3=\rho_4=1$, 
and the following:

\begin{center}
\begin{tabular}{|c|c|c|c|c|}
\hline
subcase & 1A & 1B & 1C & 1D \\
\hline
Profile & $P(x) \equiv 0$ & unit step function &$P(x)\equiv 1$&no profile, see section~\ref{sec:1CD}\\
\hline
\end{tabular}
\end{center}

%
%
%
%

For the rest
we only consider the nontrivial cases with $\bar f >0$, and $0< \rho^\pm<1$.

\subsection{Subcase 1A: $0<\rho^-<\rho^+\le \hat \rho$} \label{sec:1A}

In this case  $\rho^+\le \hat\rho$  is an unstable asymptote as $x\to +\infty$, 
so the only possible profile on $x\ge 0$ is the constant function $P(x)\equiv \rho^+$. 
Given this as the ``initial condition'', in the next Theorem we solve the initial value problem
and construct a unique 
stationary wave profile on $x<0$.

\begin{theorem}\label{th3}
Given $V^\pm,\rho^\pm$ as in subcase 1A.
There exists a unique monotone  profile $P$ 
which satisfies~\eqref{eq:dPx}-\eqref{eq:asymp}.
We have $P(x)=\rho^+$ on $x\ge 0$. 
\end{theorem}

\begin{proof} 
Given initial condition $P(x)\equiv \rho^+$ on $x\ge 0$, 
the equation~\eqref{eq:dPx} can be solved backward in $x$, 
as an initial value problem, for $x<0$. 
The existence and uniqueness of the solution for this initial value problem
can be established using the method of steps~\cite{MR0477368, MR0141863}
for delay differential equations.

\textbf{Step 1. Method of steps.} First we note that $P'(0) >0$ since $V^->V^+$. 
Now, consider the intervals $I_k = [(k-1)\ell, k \ell]$ with $k=0,-1,-2,\cdots$.
Fix an index $k$. Assume that the solution $P$ is given 
on $x\ge k\ell$, and we extend the solution on $I_k$. 
We claim that, if  
\begin{equation} \label{eq:m}
P'(x) \ge 0, \quad \rho^+\ge P(x) > \rho^-, 
\end{equation}
holds for all $x\ge k\ell$,   then \eqref{eq:m} also holds for $x\in I_k$.

Indeed, fix an $x\in I_k$.  Since the right hand side of~\eqref{eq:dPx} 
is a given Lipschitz continuous function, the existence and uniqueness of the solution
follow from standard theory on scalar ODE.  
It remains to establish the desired properties. 
Since $V$ is monotone decreasing,  $P$ is monotone increasing and 
$\phi$ is a monotone decreasing function, therefore $x\mapsto V \phi(P)$
is monotone decreasing. 
We have 
\[
v^*(L^P(x);P^\ell_{\{x\}}) \le v^*(x;P^\ell_{\{x\}}),
\]
and thus using~\eqref{eq:dPx} we conclude 
$P'(x) \ge 0$ for $x\in I_k$.

In order to establish  $\rho^-$ as a lower  bound for $P$ on $I_k$, we use contradiction.
We assume that there exists a $y\in I_k$ such that 
\[
   P(y) =\rho^- \qquad\mbox{and} \quad P(x) > \rho^- \quad \forall x>y.
\]
Then we have $v^*(x;P^\ell_{\{x\}})>V^- \phi(\rho^-)$ for any $x>y$. 
The time it takes for the car at $y$ to reach its leader is
\[
t_p(y) = \int_y^{y+\ell/\rho^-} \frac{1}{v^*(x;P^\ell_{\{x\}})} dx 
> \int_y^{y+\ell/\rho^-} \frac{1}{V^- \phi(\rho^-)} dx
= \frac{\ell}{\rho^- V^- \phi(\rho^-)} 
=\frac{\ell}{\bar f}.
\]
But according to Lemma~\ref{lm:ASF} we must have $t_p(y)=\ell/\bar f$, 
a contradiction.
We conclude that $P(x) > \rho^-$ on $I_k$.

Applying  the argument repeatedly on $k=0,-1,-2,\cdots$, by induction 
we conclude that there exists a unique monotone profile $P$ on $x<0$
satisfying~\eqref{eq:m}. 


\textbf{Step 2. Asymptotic condition.}
It remains to establish the asymptotic condition $\lim_{x\to-\infty} P(x) =\rho^-$.
Since the function $P$ is monotone and bounded, 
the limit $\tilde \rho^-\doteq \lim_{x\to-\infty} P(x)$
exists. 
By Lemma~\ref{lm:ASF}, $\tilde \rho^- $ must satisfy 
$V^- f(\tilde\rho^-) = \bar f = V^- f(\rho^-) $, where both $\tilde\rho^-$ and $\rho^-$
are less than $\hat\rho$. Since $f$ is monotone on $\rho<\hat\rho$, 
we conclude that $\tilde\rho^-=\rho^-$.
This complete the proof.
\end{proof}

\paragraph{Stability issue.}
Since $\rho^+ \le \hat\rho$ is an unstable asymptote, 
the constant function $P(x)=\rho^+$ on $x>0$ 
does not attract nearby solutions of the FtLs model,
therefore the profile is not the time asymptotic limit 
for solutions of the FtLs model. 
This is further supported by the following numerical results.

\paragraph{Sample profile and numerical simulations for the FtLs model.}
For all the numerical simulations in this paper, we use the following parameters and functions
\begin{equation}\label{eq:simp}
\ell=0.05, \quad h=0.5, \quad \phi(\rho)=1-\rho \quad \rho\in[0, 1], 
\quad w(s) = \frac{2}{h}-\frac{2}{h^2}s \quad s\in(0,h].
\end{equation}
For subcase 1A, we use the parameters
\begin{equation}\label{eq:simp1A}
V^-=2, \quad V^+=1, \quad \bar f =3/16, \quad dz=0.0002.
\end{equation}
A typical profile $P$ for subcase 1A is given in Figure~\ref{fig:1A} (left plot). 
We observe that the profile is  monotone and  continuously differentiable.

A numerical simulation is also performed for the FtLs model~\eqref{FtLs}, 
with the following Riemann-like initial data
\begin{equation}\label{eq:z1B}
z_i(0) = \begin{cases} i \frac{\ell}{\rho^+}, \quad & i\ge 0,\\ 
i \frac{\ell}{\rho^-} \quad & i< 0, \end{cases}
\qquad 
\mbox{such that} \quad 
\rho_i(0) = \begin{cases}\rho^+, \quad & z_i(0)\ge 0,\\ 
\rho^- \quad & z_i(0)< 0. \end{cases}
\end{equation}
Numerical solutions for the FtLs model for $0\le t \le T_f=4$ is computed.
Let $\tau_p=\frac{\ell}{f^+(\rho^+)}$.  
The graphs for the mappings $z_i(t)\mapsto \rho_i(t)$ 
for $t\in[T_f-\tau_p,T_f]$ are plotted in green in Figure~\ref{fig:1A} (right plot)
together with the points $\{ z_i(T_f),\rho_i(T_f)\}$ in red. 
We observe that, as time grows, some oscillation enters the region 
$x>0$, where the profile $P(x)=\rho^+$ is constant. 
Since $\rho^+ < \hat\rho$ is not a stable asymptote,
the oscillation persists as time grows, 
and the solution does not approach the stationary  profile $P$.
This behavior is typical for subcase 1A and is independent on the parameters
($\ell, dz, \bar f, V^-, V^+$, etc.), as long as we are in subcase 1A.

\begin{figure}[htbp]
\begin{center}
\includegraphics[height=3.2cm,clip,trim=2mm 1mm 7mm 6mm]{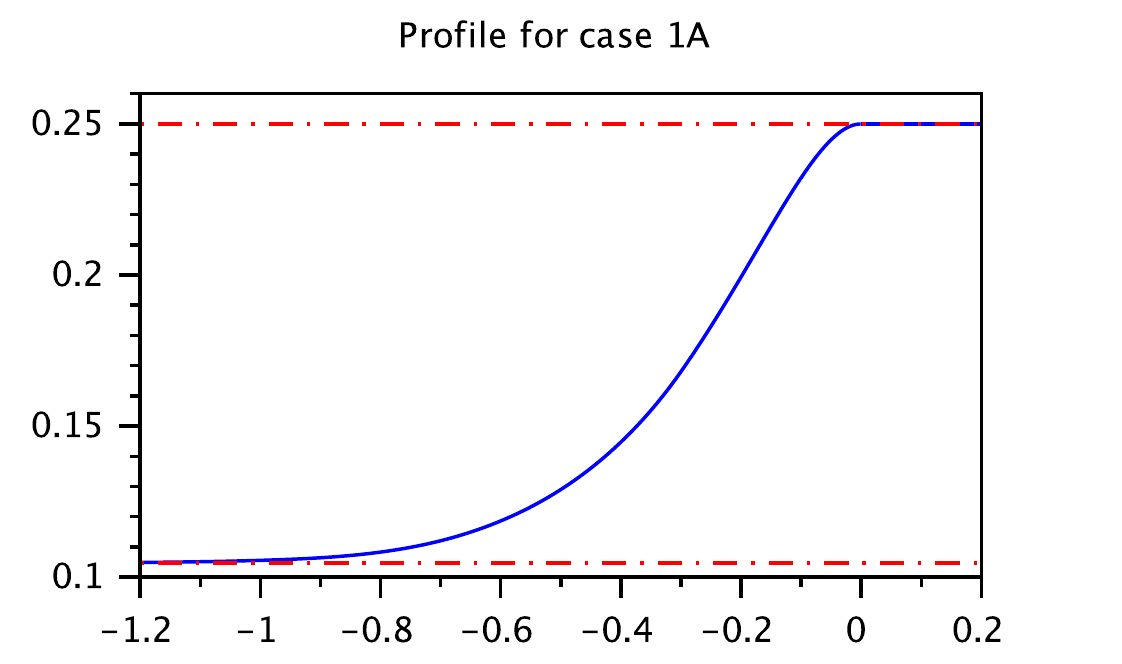}
\includegraphics[height=3.2cm,clip,trim=2mm 1mm 12mm 6mm]{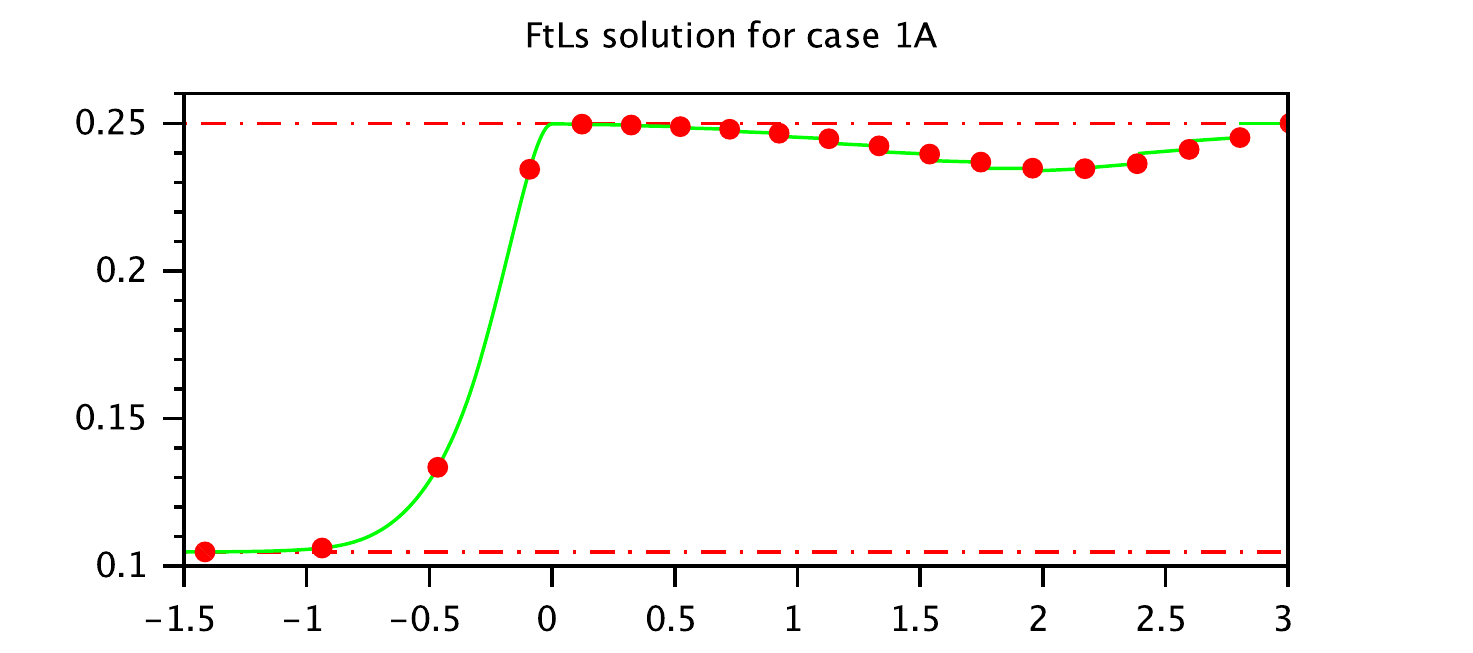}
\caption{Left: Typical profile $P$ for case 1A. Right: Solution of the FtLs model with Riemann-like initial condition.}
\label{fig:1A}
\end{center}
\end{figure}

\subsection{Subcase 1B:  with $0<\rho^- < \hat\rho< \rho^+<1$}
\label{sec:1B}

For this subcase we have $\rho^- < \hat\rho$ and $\rho^+>\hat\rho$, 
where both $\rho^-,\rho^+$ are stable asymptotes. 
We will show that
there are infinitely many monotone stationary wave profiles.

\begin{theorem}[Existence of profiles]\label{th1}
Given $V^\pm,\rho^\pm$ as in subcase 1B.
There exist infinitely many   monotone stationary wave profiles $P$ 
which satisfies the equation~\eqref{eq:dPx} and the asymptotic 
conditions~\eqref{eq:asymp}. 
\end{theorem}

\begin{proof}
On $x\ge0$, we have $V(x) \equiv V^+$, therefore the profile 
$P$ must be either constant $P(x)\equiv \rho^+$  
or some horizontal shift of $W$,
 i.e., the stationary profile established in~\cite{RidderShen2018}. 
The profile $P$ is smooth and monotone, taking values between 
$\rho_2$ and $\rho^+$, where $\rho_2$ is defined in~\eqref{eq:rhoss}. 
%
In particular, it holds that $\rho_2 < P(0) \le \rho^+$. 

Once the profile is given on $x\ge 0$, we solve an initial value problem backward in $x$.
By the same argument as in the proof of Theorem~\ref{th3}, 
there exists a unique monotone solution for each initial value problem, 
which satisfies also the asymptotic condition $\lim_{x\to -\infty}P(x) =\rho^-$. 
Therefore there exist infinitely many stationary wave profiles.
\end{proof}

\textbf{Local stability of the profiles.}
Recall the definitions of $\rho_1, \rho_2, \rho_3,\rho_4$ in~\eqref{eq:rhoss}. 
Let $P^\sharp$ denote the profile with $P(x) = \rho^+=\rho_3$ on $x\ge 0$,
as in Theorem~\ref{th1}. 
Let $P^\flat$ denote the profile with $P(x)= \rho_2$ on $x\ge 0$, as in 
Theorem~\ref{th3}.
We define the region $D$ as:
\begin{equation}\label{eq:D}
D \defeq \left\{ (x,\rho) : P^\flat(x) < \rho \le P^\sharp(x), ~ x\in\mathbb{R} \right\}.
\end{equation}

The next Theorem shows that $D$ is a basin of attractions of the stationary profiles.

\begin{theorem}[Local stability of the profiles]\label{th2}

Let $\{ z_i(t),\rho_i(t)\}$ be the solution of the FtLs model~\eqref{FtLs}, with initial
condition $\{z_i(0),\rho_i(0)\}$ which satisfies
\begin{equation}\label{eq:k1}
(z_i(0), \rho_i(0)) \in D \qquad  \forall i\in\mathbb{Z}.
\end{equation} 
Then, we have
\begin{equation}\label{eq:k1t}
(z_i(t), \rho_i(t)) \in D \qquad  \forall i\in\mathbb{Z}, t\ge 0.
\end{equation} 
Furthermore, there exists a stationary wave profile $\tilde P$, 
whose graph lies between $P^\flat$ and $P^\sharp$,
such that 
\begin{equation}\label{eq:k2}
\lim_{t\to\infty} \left[\tilde P(z_i(t)) -  \rho_i(t) \right] =0, \qquad \forall i\in\mathbb{Z}.
\end{equation} 
\end{theorem}

\begin{proof} \textbf{(1).}
By Lemma~\ref{lm:order}, all stationary profiles in Theorem~\ref{th1} never cross each other. 
Then, any point $(x,\rho)\in D$ lies on a unique stationary profile. 
In other words, the region $D$ can be parametrized by the profiles $P$. 
This motivates the definition of a function
\begin{equation}\label{eq:Psi}
\Psi(x,\rho) \defeq P(0), \quad \mbox{where } (x,\rho)\in D ~ \mbox{and $P$ is the profile with}~P(x)=\rho.
\end{equation} 

Let $\{z_i(t),\rho_i(t)\}$ be the solution of the FtLs model, and define 
\begin{equation}\label{eq:Psi_i}
\Psi_i(t) \defeq \Psi(z_i(t),\rho_i(t)), \qquad \forall i\in\mathbb{Z}, \forall t\ge 0.
\end{equation} 

\textbf{(2).}   Fix a time $t\ge 0$ 
and assume that $\{z_i(t), \rho_i(t)\} \in D$  for all $i \in\mathbb{Z}$.
Let $J$ be the index such that 
\begin{equation}\label{eq:Ps2}
\Psi_J(t) \ge \Psi_i(t) \quad \forall i\in\mathbb{Z} \qquad \mbox{and} \quad
\Psi_J(t) > \Psi_i(t) \quad \forall i >J.
\end{equation}
Assume that $J$ is finite, and 
let $\hat P$ be the profile that satisfies $\hat P(z_J(t)) = \rho_J(t)$. 
By~\eqref{eq:Ps2} we have
\begin{equation}\label{eq:Ps3}
\hat P(z_i) > \rho_i \qquad \forall i > J. 
\end{equation}

We claim that 
\begin{equation}\label{eq:Ps4}
\dot\Psi_J(t) < 0,\qquad \mbox{i.e.,}\qquad 
\hat  P'(z_J) > \frac{\dot \rho_J}{\dot z_J}.
\end{equation}
Indeed, let $\{\hat z_i\}$ be the unique car distribution generated by $\hat P$ 
with $\hat z_J=z_J$, and
let $\hat P^\ell_{\{z_J\}}$ be the piecewise constant function generated by $\hat P$ 
according to~\eqref{eq:Pell2}. 
We also let $\rho^\ell(t,\cdot)$ be the piecewise function defined in~\eqref{eq:rhoell}. 
By~\eqref{eq:dPx}, \eqref{FtLs} and \eqref{rhodot} we have 
\begin{eqnarray*}
\hat P'(z_J) &=& \frac{\rho^2_J}{\ell \cdot v^*(z_J;\hat P^\ell_{\{z_J\}})} \Big[v^*(z_J;\hat P^\ell_{\{z_J\}})-
v^*(L^{\hat P}(z_J);\hat P^\ell_{\{z_J\}})\Big],\\
\frac{\dot \rho_J}{\dot z_J} &=& \frac{\rho^2_J}{\ell \cdot v^*(z_J;\rho^\ell)} \Big[v^*(z_J;\rho^\ell)-
v^*(z_{J+1};\rho^\ell)\Big].
\end{eqnarray*} 
We compute
\begin{eqnarray*}
I_1&\defeq &\left[ v^*(z_J;\hat P^\ell_{\{z_J\}})  - v^*(z_{J+1};\hat P^\ell_{\{z_J\}}) \right] - \left[ v^*(z_J;\rho^\ell )  - v^*(z_{J+1};\rho^\ell)\right] \\
&=& \left[ v^*(z_J;\hat P^\ell_{\{z_J\}})  - v^*(z_J;\rho^\ell )  \right] - \left[ v^*(z_{J+1};\hat P^\ell_{\{z_J\}}) - v^*(z_{J+1};\rho^\ell)\right] \\
&=& \int_{z_J}^{z_J+h} V(y) \left[ \phi(\hat P^\ell_{\{z_J\}}(y))- \phi(\rho^\ell(t,y))\right] w(y-z_J)\; dy\\
&& - \int_{z_{J+1}}^{z_{J+1}+h} V(y) \left[ \phi(\hat P^\ell_{\{z_J\}}(y))- \phi(\rho^\ell(t,y))\right] w(y-z_{J+1})\; dy.
\end{eqnarray*}
Since $\hat P(z_J)=\rho_J$,  we have $z_{J+1} = \hat z_{J+1}$.  
By~\eqref{eq:Ps3} we also have 
\[ 
\hat P^\ell_{\{z_J\}}(x) = \rho^\ell(t,x) \quad \forall x\in[z_J, z_{J+1}), \qquad 
\hat P^\ell_{\{z_J\}}(x) > \rho^\ell(t,x) \quad \forall x > z_{J+1}. 
\]
We arrive at
\begin{eqnarray*}
I_1 &=& \int_{z_{J+1}}^{z_{J+1}+h} V(y)  \left[ \phi(\hat P^\ell_{\{z_J\}}(y))- \phi(\rho^\ell(t,y))\right] 
\cdot \left[ w(y-z_J)-w(y-z_{J+1})\right]\; dy.
\end{eqnarray*}
Since $\phi'<0$, we have
\[
\phi(\hat P^\ell_{\{z_J\}}(y)) - \phi(\rho^\ell(t,y))  <0 \quad \forall y > z_{J+1}. 
\]
Recall also that $w'<0$ on its support $[0,h]$, then we have
\[
 w(y-z_J)-w(y-z_{J+1})  <0 \quad \forall y\in [z_{J+1},z_{J+1}+h).
\]
Therefore we conclude that 
$I_1>0$, which implies~\eqref{eq:Ps4}. 


\textbf{(3).} A completely symmetric result holds for the minimum.
Let $\hat J$ be the index such that
\begin{equation}\label{eq:Ps2min}
\Psi_{\hat J}(t) \le \Psi_i(t) \quad \forall i\in\mathbb{Z} \qquad \mbox{and} \quad
\Psi_{\hat J}(t) < \Psi_i(t) \quad \forall i >{\hat J}.
\end{equation}
Assume $\hat J$ is finite, and
let $\widehat P$ be the profile that satisfies $\widehat P(z_{\hat J}(t)) = \rho_{\hat J}(t)$. 
Then
\begin{equation}\label{eq:Ps5}
\dot\Psi_{\Hat J}(t) > 0,\qquad \mbox{i.e.,}\qquad 
\widehat P'(z_{\Hat J}) < \frac{\dot \rho_{\Hat J}}{\dot z_{\Hat J}}.
\end{equation}

\textbf{(4).}
We note that,  if $\{z_i(0), \rho_i(0)\} \in D$ for all $i \in\mathbb{Z}$, 
then~\eqref{eq:Ps4} and~\eqref{eq:Ps5}  imply that 
 $\{z_i(t), \rho_i(t)\} \in D$  for all $i \in\mathbb{Z}$ and  $t\ge 0$.
Furthermore, if $D$ is not empty, then at least one of $J,\hat J$ is finite,
we have 
\[
\lim_{t\to\infty} \left[ \max_{i\in\mathbb{Z}}\{\Psi_i(t)\} 
- \min_{i\in\mathbb{Z}}\{\Psi_i(t)\} \right] = 0,
\]
which further implies~\eqref{eq:k2}, completing the proof.
\end{proof}

\paragraph{Sample profiles and numerical simulations for the FtLs model.} 
We use the same parameters and functions in~\eqref{eq:simp}-\eqref{eq:simp1A}.
Sample profiles for subcase 1B are given in Figure~\ref{fig:1B} (left plot).
We observe that on $x>0$, the profiles match the horizontally shifted versions
of $W$.  The profiles are continuous and monotone.

Numerical simulations for the FtLs model~\eqref{FtLs} are performed with the following
Riemann-like initial data
\begin{equation}\label{eq:sim1A}
z_i(0) = \begin{cases} i \frac{\ell}{\rho^+} +c_0, \quad & i\ge 0,\\ 
i \frac{\ell}{\rho^-}+c_0 \quad & i< 0, \end{cases}
\qquad 
\mbox{such that} \quad 
\rho_i(0) = \begin{cases}\rho^+, \quad & z_i(0)\ge c_0,\\ 
\rho^- \quad & z_i(0)< c_0. \end{cases}
\end{equation}
Here $c_0\in\mathbb{R}$ is a small number, which we can vary 
to simulate various cases.
Numerical solutions for the FtLs model at $T_f=4$ is computed,
for 
\[ c_0 = \gamma \cdot \ell/\rho^-, \qquad \gamma = \{ -0.1, 0.1, 0.5, 1\}. \]
Similar to subcase 1A, we plot the mappings $z_i(t)\mapsto \rho_i(t)$ 
for $t\in[T_f-\tau_p,T_f]$ in green in Figure~\ref{fig:1A} (right plot),
together with the points $\{ z_i(T_f),\rho_i(T_f)\}$ in red. 
We observe that, after a  short time, the solutions
$\{z_i(t),\rho_i(t)\}$  approach some stationary wave profile $P$, 
confirming the stability result in Theorem~\ref{th2}.

\begin{figure}[htbp]
\begin{center}
\includegraphics[height=3.2cm,clip,trim=1mm 2mm 10mm 7mm]{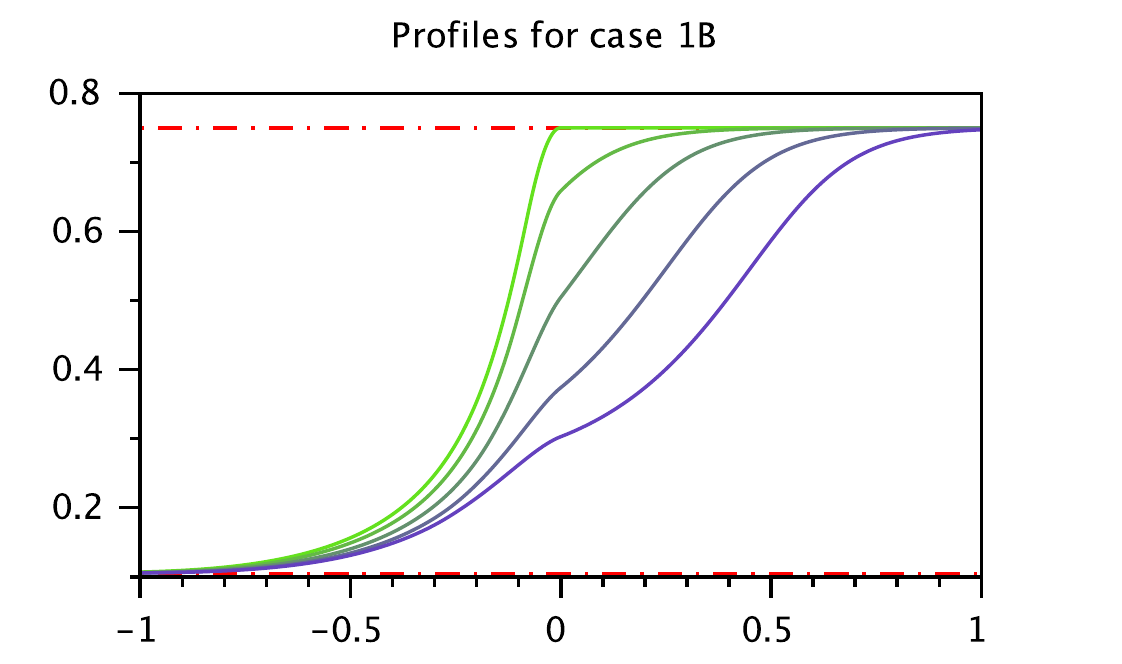}
$\quad$
\includegraphics[height=3.2cm,clip,trim=1mm 2mm 10mm 7mm]{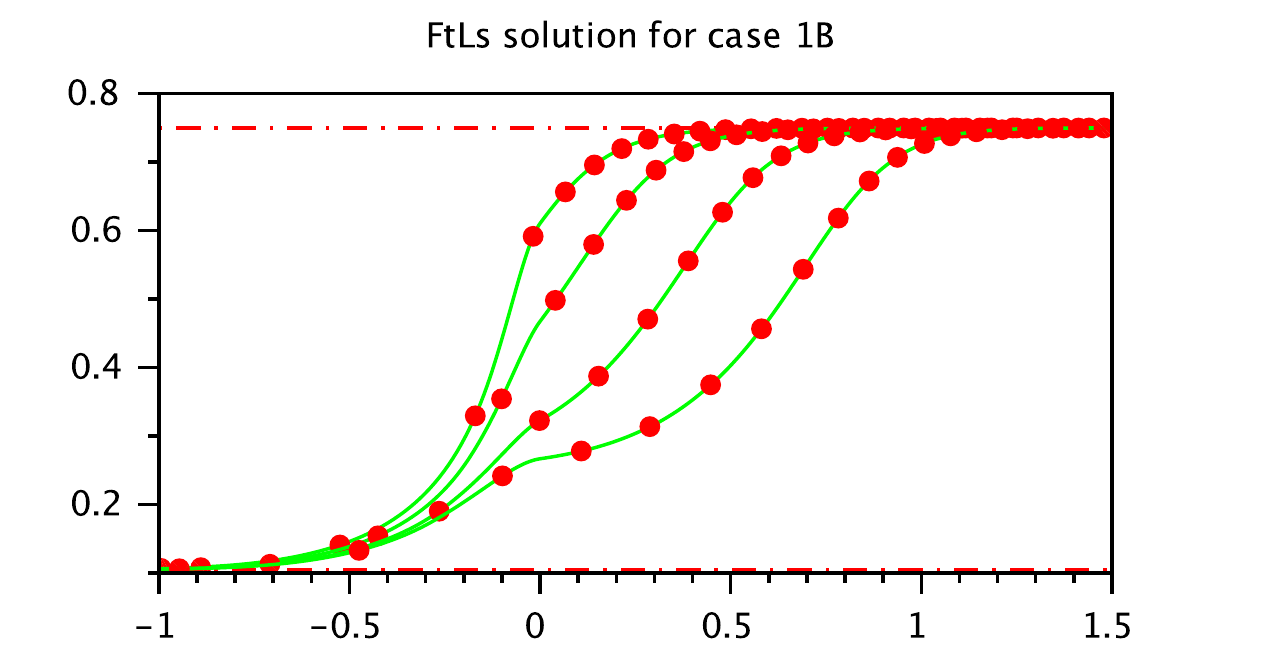}
\caption{Left: Sample profiles $P$ for subcase 1B. Right: Solution of the FtLs model with Riemann-like initial conditions for subcase 1B.}
\label{fig:1B}
\end{center}
\end{figure}

\subsection{Subcase 1C  $\hat\rho \le \rho^+ <\rho^-<1$ 
and subcase 1D  $0 < \rho^+ <\hat\rho< \rho^- < 1$} 
\label{sec:1CD}

For subcase 1C,  $\rho^- > \hat\rho$ is not a stable asymptote for $x\to-\infty$.
For subcase 1D,  $\rho^+ < \hat\rho, \rho^- > \hat\rho$ are not a stable asymptotes 
for $x\to-\infty$.
Therefore there are no stationary profiles for either subcases. 

\paragraph{Numerical simulations.}
We perform numerical simulation for subcases 1C and 1D,  
for the FtLs model~\eqref{FtLs} with Riemann-like initial condition~\eqref{eq:z1B}.
The results are presented in Figure~\ref{fig:1C1D}.
For each subcase, we plot the mappings $z_i(t)\mapsto \rho_i(t)$ for $t\in[T_f-\tau_p,T_f]$
in green, as well 
the points $\{ z_i(T_f),\rho_i(T_f)\}$ in red.
For subcase 1C, we observe that oscillations enter into the region $x<0$ as
time grows, and they persist in time. 
For subcase 1D, we observe that oscillations enter both regions $x<0$ and $x>0$
as $t$ grow.

\begin{figure}[htbp]
\begin{center}
\includegraphics[height=3.2cm,clip,trim=1mm 2mm 8mm 7mm]{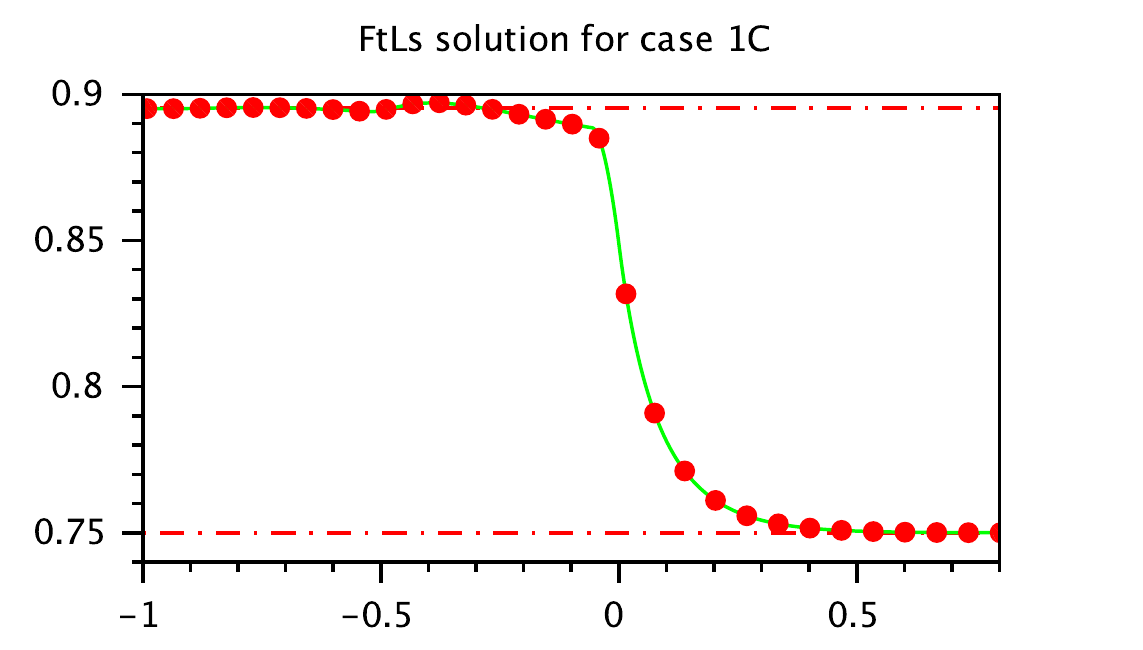}$\quad$
\includegraphics[height=3.2cm,clip,trim=0mm 2mm 10mm 7mm]{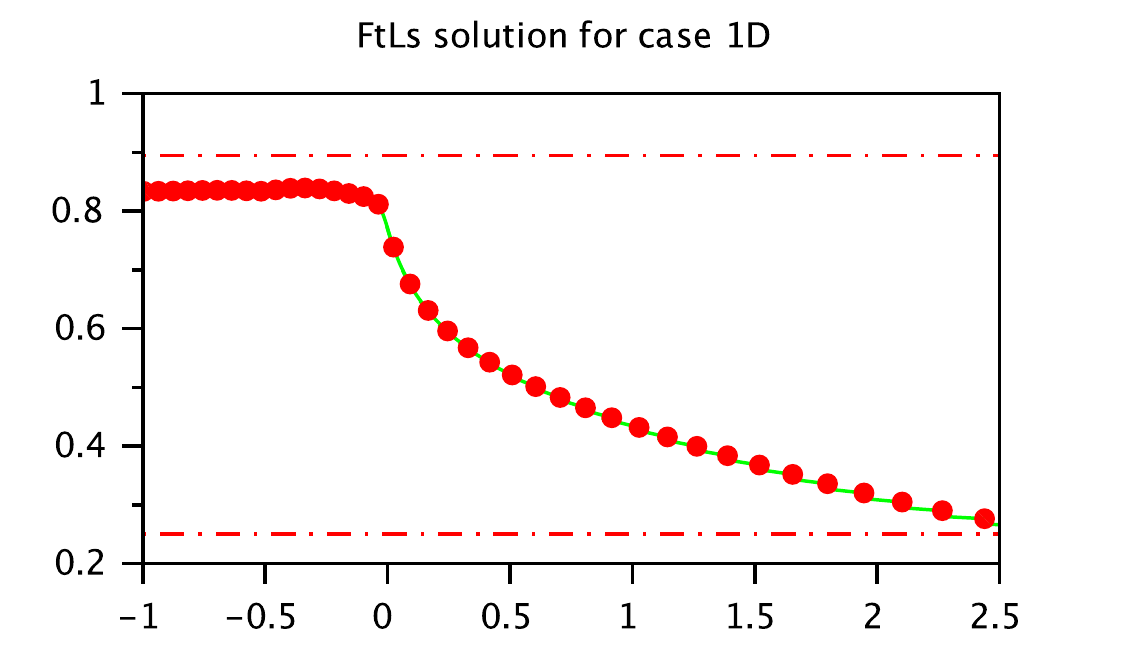}
\caption{Solution of the FtLs model with Riemann-like initial condition, for subcase 1C (left plot) 
and subcase 1D (right plot).}
\label{fig:1C1D}
\end{center}
\end{figure}

\section{Case 2. $V^-<V^+$}
\label{sec:case2}
\setcounter{equation}{0}
In this section we establish results on stationary wave profiles for the case $V^- < V^+$.
The structure and some details of the analysis are similar to case 1, 
therefore we skip the repetitive details and emphasize the differences. 
%
%
Fix a value $\bar f$, we set 
(see Figure~\ref{fig:F2})
\begin{equation}\label{defs:pf}
0\le \rho_1 < \rho_2\le \hat\rho \le \rho_3 <\rho_4\le 1,
\qquad 
f^+(\rho_1)=f^+(\rho_4)=
f^-(\rho_2)=f^-(\rho_3)=\bar f.
\end{equation}

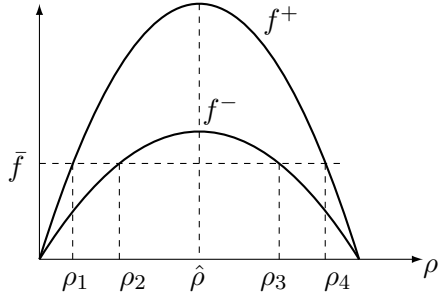
\begin{figure}[htbp]
\begin{center}
\setlength{\unitlength}{0.85mm}
\begin{picture}(60,50)(-3,-5)  
\put(0,0){\vector(1,0){60}}\put(60,-2){$\rho$}
\put(0,0){\vector(0,1){40}}
\multiput(0,15)(2,0){24}{\line(1,0){1}}\put(-5,13){$\bar f$}
\multiput(5.2,15)(0,-2){8}{\line(0,-1){1}}\put(2,-4){ $\rho_1$}
\multiput(12.5,15)(0,-2){8}{\line(0,-1){1}}\put(11,-4){ $\rho_2$}
\multiput(25,40)(0,-2){20}{\line(0,-1){1}}\put(22,-4){ $\hat\rho$}
\multiput(37.5,15)(0,-2){8}{\line(0,-1){1}}\put(33,-4){ $\rho_3$}
\multiput(44.7,15)(0,-2){8}{\line(0,-1){1}}\put(43,-4){ $\rho_4$}
\put(25.5,22){$f^-$}\put(35,36){$f^+$}
\thicklines
\qbezier(0,0)(25,40)(50,0)
\qbezier(0,0)(25,80)(50,0)
\end{picture}
\caption{Graphs of the functions $f^-, f^+$, 
and locations of $\rho_1,\rho_2,\rho_3,\rho_4$ and $\hat\rho$.}
\label{fig:F2}
\end{center}
\end{figure}

We consider the nontrivial case $\bar f>0$, in 4 subcases:

\begin{center}
\begin{tabular}{|c|c|c|c|c|}
\hline
subcase & 2A & 2B & 2C & 2D \\
\hline
$(\rho^-,\rho^+)$ & $(\rho_2,\rho_1)$ & $(\rho_2,\rho_4)$&$(\rho_3,\rho_4)$&$(\rho_3,\rho_1)$\\
\hline
\end{tabular}
\end{center}

%
%


\subsection{Subcase 2A: $0<\rho^+ <\rho^-\le \hat \rho$} 

Since $\rho^+ < \hat\rho$,  $\rho^+$ is an unstable asymptote as $x\to+\infty$. 
Similar to subcase 1A, 
the only possible solution on $x\ge 0$ is the constant function
$P(x) \equiv \rho^+$. 
Following  similar arguments as in the proof for Theorem~\ref{th3}, 
there exists  a unique stationary wave profile.

\begin{theorem}[Existence of a unique profile]\label{th5}
Given $V^\pm,\rho^\pm$ as in subcase 2A. 
There exists a unique  stationary wave profile $P$ 
which satisfies~\eqref{eq:dPx}-\eqref{eq:asymp}. 
The profile  is constant on $x\ge 0$, and monotone decreasing on $x<0$. 
\end{theorem}

A sample profile can be found in Figure~\ref{fig:2A} (left plot), 
using the same parameters and functions in~\eqref{eq:simp}-\eqref{eq:simp1A},
except for $V^-=1, V^+=2$. 
Similar to subcase 1A, since $\rho^+ < \hat\rho$ is an unstable asymptote
as $x\to\infty$, the profile $P$ does not attract nearby solutions 
of the FtLs model~\eqref{FtLs}. 
This is confirmed by the numerical evidence in Figure~\ref{fig:2A} (right plot)
for the solutions of~\eqref{FtLs}.

\begin{figure}[htbp]
\begin{center}
\includegraphics[height=3.2cm,clip,trim=1mm 2mm 8mm 7mm]{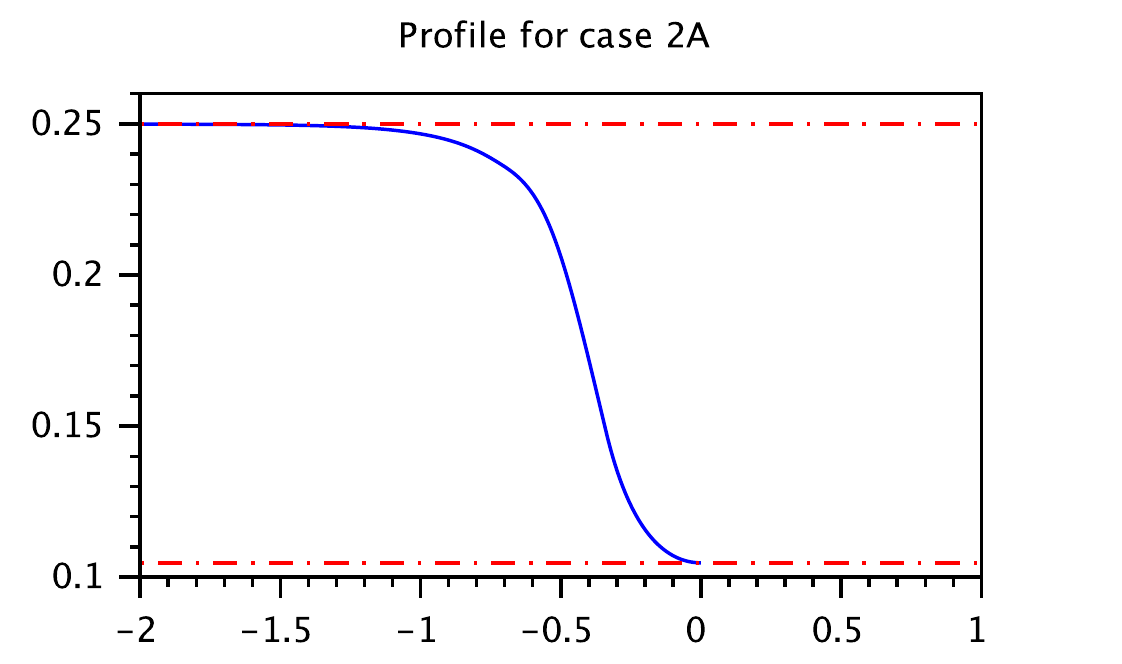}
$\qquad$
\includegraphics[height=3.2cm,clip,trim=1mm 2mm 10mm 7mm]{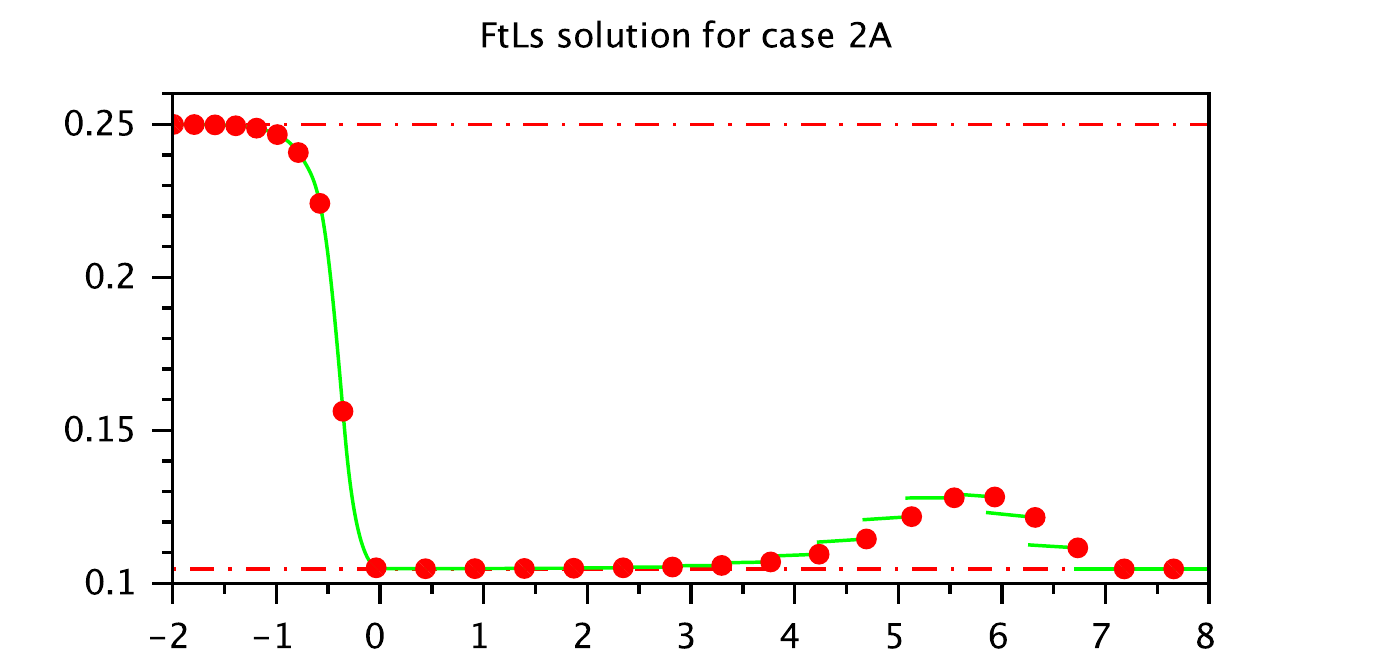}
\caption{Left: Typical profile $P$ for case 2A. Right: Solution of the FtLs model with Riemann-like initial condition.}
\label{fig:2A}
\end{center}
\end{figure}

\subsection{Subcase 2B:  with $0<\rho^- < \hat\rho< \rho^+<1$}

This is the counterpart for subcase 1B. 
Here both $\rho^-$ and $\rho^+$ are stable asymptotes at $x\to-\infty$ and $x\to\infty$ respectively. 
Similar to subcase 1B,  there are infinitely many stationary wave profiles.
However, due to the upward jump in $V$ at $x=0$, the profiles are no longer monotone,
resulting in more involving analysis. 

\begin{theorem}[Existence of profiles]\label{th5-2}
Given $V^\pm,\rho^\pm$ as in subcase 2B.
There exist infinitely many  stationary wave profiles $P$ 
which satisfies~\eqref{eq:dPx}-\eqref{eq:asymp}. 
These profiles are monotone on $x>0$, but might not be monotone on $x<0$. 
\end{theorem}

\begin{proof}
\textbf{1.} On $x\ge 0$, $P$ must match some horizontal shift of $W$. 
Recall the definition~\eqref{defs:pf}. 
For any profile we have $P(x) > \rho_1$ on $x>0$.

Let $P^\flat$ be the unique  profile in Theorem~\ref{th5} for subcase 2A, 
with $P^\flat(x) = \rho_1$  on $x\ge 0$. 
The profile is monotone decreasing and satisfies the asymptotic condition 
$\lim_{x\to-\infty} P^\flat(x) =\rho^-$.
Furthermore, since $W(x) > \rho_1 $ on $x\ge 0$,
by the ordering property in Lemma~\ref{lm:order} we conclude that
all profiles $P$ will lie above $P^\flat$ also on  $x<0$. 
Thus, $P^\flat$ serves as a lower envelope for all profiles.

\textbf{2.}  
Consider the interval $I=[z^\flat, 0]$ where $L^P(z^\flat)=-h$. 
Using horizontally shifted versions of $W$ as initial condition for $P$ on $x\ge 0$,
we get various solutions of $P$ on $I$. 
By continuity there exist infinitely many profiles $P$ such that 
$P(x) < \rho^-$ for all $x\in I$, and $P$ is monotone decreasing on $I$. 
By a similar proof as for Theorem~\ref{th5},
one concludes the existence of infinitely many monotone stationary wave profiles on $x<0$.

\textbf{3.} 
It remains to show that, if a profile $P$ lies between $\rho^-$ and $\rho_3$ on $x<z^\flat$,
then we must have 
\begin{equation}\label{eq:LLL}
\lim_{x\to-\infty} P(x) = \rho^-.
\end{equation}

Indeed, we first observe that, if $P$ approaches a limit as $x\to-\infty$, 
then it must be either $\rho^-$ or $\rho_3$
according to the periodic property. 
If in addition $P$ is monotone, then since $\rho_3$ is an unstable asymptote,
then $P$ must be monotone increasing on $x<0$, and we conclude~\eqref{eq:LLL}.


\textbf{4.}
We are left  to consider  the case that $P$ is oscillatory
on $x<z^\flat$ between $\rho^-$ and $\rho_3$.
By the periodic property we derive that
\begin{equation}\label{eq:sss2}
\frac{P(x)}{\ell} \int_x^{x+\ell/P(x)} \frac{1}{v^*(z;P^\ell_{\{z\}})} dz= \frac{P(x)}{\bar f} 
= \frac{f(P(x))}{f(\rho^-)} \cdot \frac{1}{V^- \phi(P(x))}.
\end{equation}
Since $f$ is strictly concave, for any $\rho\in [\rho^-,\rho_3]$ we have the estimate
\begin{equation}\label{eq:sss3}
\frac{f(\rho)}{f(\rho^-)} = 1+ \frac{f(\rho)-f(\rho^-)}{\bar f} \ge 1+ C_f (\rho)
\end{equation}
where $C_f$ is a positive function, defined as 
\[
C_f (\rho) \defeq  \frac{1}{\bar f} \cdot  \min\left\{(\rho-\rho^-), (\rho_3 - \rho) \right\} \cdot
\min\left\{ \frac{f(\hat\rho)-f(\rho^-)}{\hat\rho-\rho^-}, \frac{f(\hat\rho)-f(\rho_3)}{\hat\rho-\rho_3}\right\}.
\]

By \eqref{eq:sss2}-\eqref{eq:sss3} we now have
\[
\frac{1}{V^-\phi(P(x))} \left(1+ C_f (P(x))\right)
\le
\mbox{average}_{y\in[x,L^P(x)]} \left\{\frac{1}{v^*(y;P^\ell_{\{y\}})} \right\}.
\]
Since $\phi$ is a monotone function, this implies that,
\begin{equation}\label{eq:sss}
\forall x<z^\flat, ~~\exists y\in(x,L^P(x)+h)~\mbox{ s.t. }~
P(x) (1+C_f(P(x)) < P(y),
\end{equation}
In particular, \eqref{eq:sss} holds also for the case where $x$ is a local maximum. 
Let $\{x_k\}$ be a sequence of local maximum such that
$x_{k+1} < x_k$ and $P(y) < P(x_k),  \forall y<x_k$.
By \eqref{eq:sss} we have that
\begin{equation}\label{eq:sss5}
P(x_{k+1}) (1+C_f(P(x_{k+1}))) \le P(x_k)\qquad \forall k,
\end{equation}
where $C_f(\rho)=0$ only when $\rho=\rho^-$. 
Thus,  these max values (if they exist) are strictly monotone. 
If  $\{x_k\}$ is a finite sequence, we conclude~\eqref{eq:LLL}.
Otherwise, if $\{x_k\}$ is an infinite sequence, 
then~\eqref{eq:sss5} implies that 
$\lim_{k\to\infty} P(x_k) = \rho^-$, concluding~\eqref{eq:LLL}.
\end{proof}


\textbf{Local Stability.}
The stability result and the proof for these profiles are 
the same as Theorem~\ref{th2} for subcase  1B,   and we omit the details. 
Sample profiles and numerical simulations for the FtLs model
are presented in Figure~\ref{fig:2B}, similar to subcase 1B. 

\medskip
\noindent\textbf{Subcases 2C and 2D.} 
Similar to the subcases 1C and 1D, there are no stationary profiles 
for subcases 2C and 2D. 
In Figure~\ref{fig:2C2D} 
we present similar numerical simulations for these subcases for 
the FtLs model~\eqref{FtLs}.

\begin{figure}[htbp]
\begin{center}
\includegraphics[height=3.2cm,clip,trim=1mm 2mm 8mm 7mm]{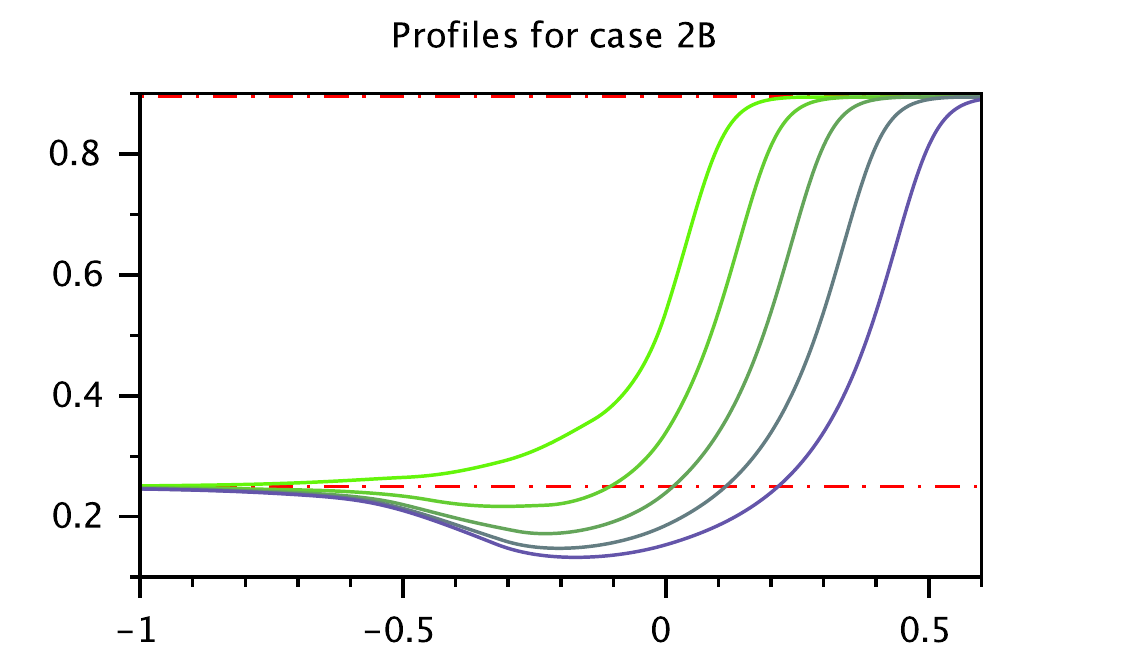}
$\qquad$
\includegraphics[height=3.2cm,clip,trim=1mm 2mm 8mm 7mm]{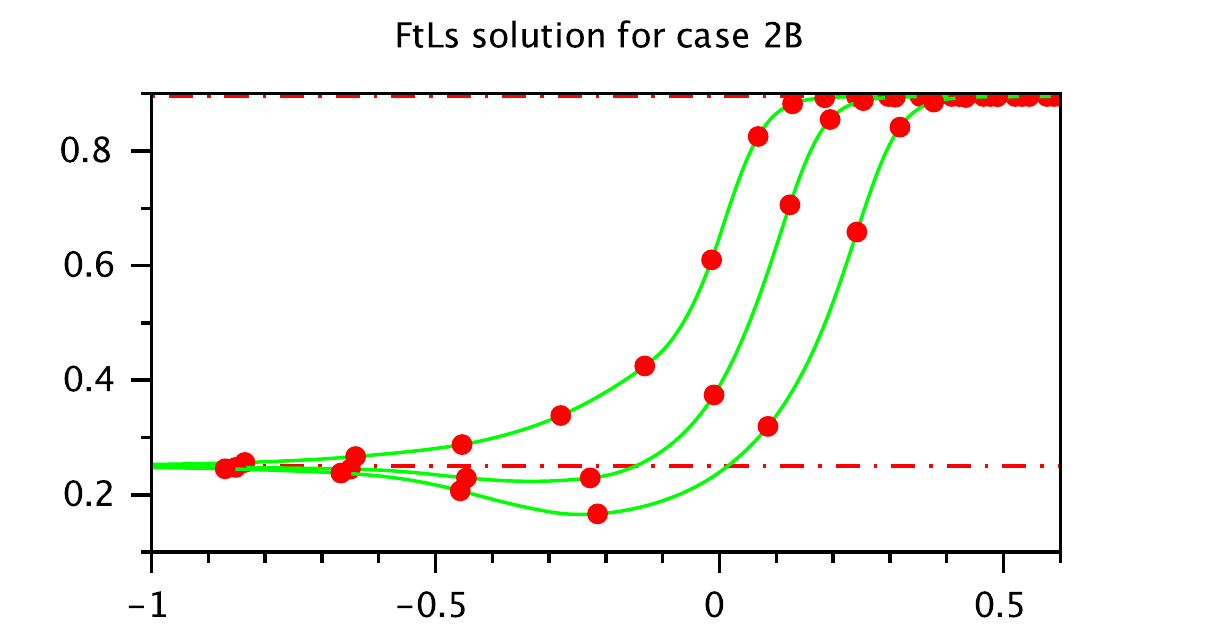}
\caption{Left: Typical profiles $P$ for subcase 2B. Right: Solution of the FtLs model with Riemann-like initial condition for subcase 2B.}
\label{fig:2B}
\mbox{}\\
\includegraphics[height=3.2cm,clip,trim=1mm 2mm 8mm 7mm]{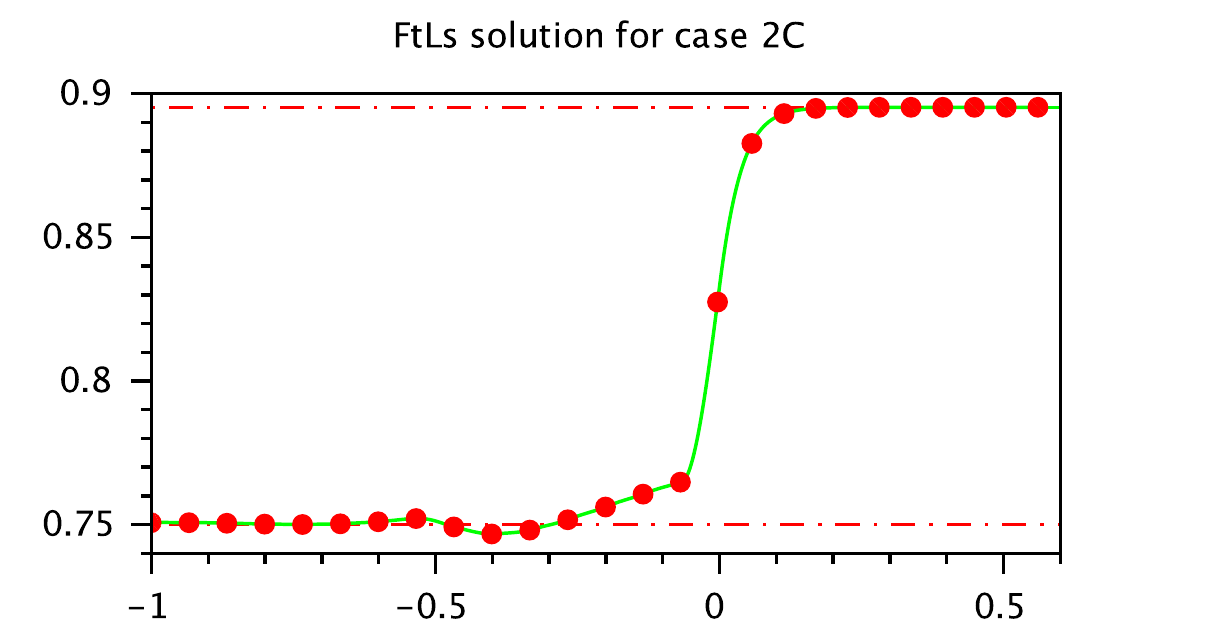}$\quad$
\includegraphics[height=3.2cm,clip,trim=1mm 2mm 8mm 7mm]{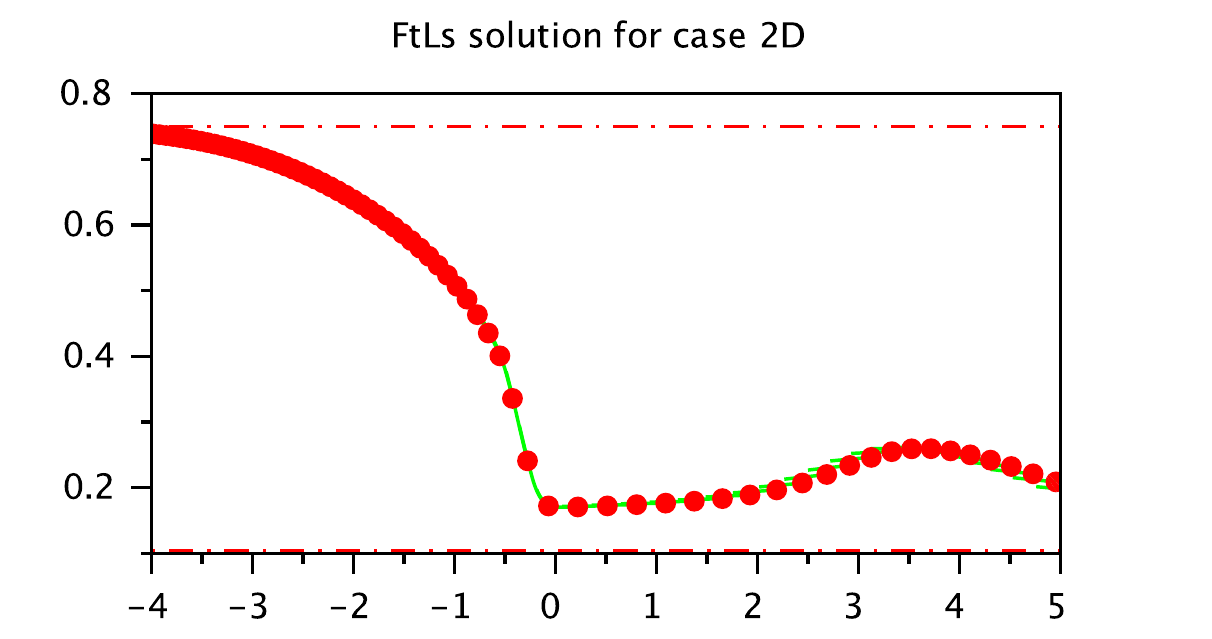}
\caption{Solution of the FtLs model with Riemann-like initial condition for subcase 2C (left plot)
and  subcase 2D (right plot).}
\label{fig:2C2D}
\end{center}
\end{figure}

\section{Convergence of stationary wave profiles to various limits}
\label{sec:conv}
\setcounter{equation}{0}

As mentioned in the introduction, when the road condition $V$ is discontinuous, 
rigorous analysis of  convergences for~\eqref{FtLs} 
to various limits as $\ell\to 0$ and/or $h\to 0$ are interesting open problems.
As a first step in this direction, we study two of these limits in the setting of 
stationary wave profiles.

\paragraph{Limit 1: Micro-macro limit.} 
Fix $h>0$.   Formally, as $\ell\to 0$,  the particle model~\eqref{FtLs} converges to 
the nonlocal PDE with discontinuous coefficient~\eqref{eq:clawNL}. 
Let $Q$ be a stationary profile for~\eqref{eq:clawNL} around $x=0$. 
Then, $Q$ satisfies the following integral equation with discontinous coefficient
\begin{equation}\label{eq:Q}
Q(x) \cdot \int_x^{x+h} V(y) \phi(Q(y)) w(y-x)\; dy = \bar f \qquad \forall x\in\mathbb{R}.
\end{equation}
With a slight abuse of notation, we denote the averaging operator as
\begin{equation}\label{eq:cAA}
\mathcal{A}(x; V, Q) \doteq \int_x^{x+h} V(y) \phi(Q(y)) w(y-x) \; dy.
\end{equation}
Taking the limit $x\to\pm\infty$ we get
\[
 \lim_{x\to\pm\infty} \mathcal{A}(x; V, Q)  
= V^- f^-(\rho^-) = V^+ f^+(\rho^+) = \bar f, 
\qquad \mbox{where}~
\rho^\pm =\lim_{x\to\pm\infty} Q(x).
\]

The stationary profiles $Q$ were studied in a recent work~\cite{ShenTR}, 
where results on existence, uniqueness, and stability are established,
in a similar setting as for the FtLs model~\eqref{FtLs}. 
For a given set of asymptotic value $\rho^\pm$, 
for subcases 1A and 2A, there exists a unique stationary wave profile $Q$, which
is not asymptotically stable. 
For subcases 1B and 2B, there 
exist infinitely many stationary wave profiles, which 
are local attractors for the solution of the Cauchy problems
of the conservation laws.
For all other subcases there are no stationary wave profiles. 
Sample profiles $Q$ for subcases 1B and 2B 
are shown in Figure~\ref{fig:Qs}, taken from~\cite{ShenTR}, for comparison.
Furthermore, when $V(x)\equiv 1$, 
the micro-macro limit of the traveling wave profiles is proved in~\cite{RidderShen2018}.

\begin{figure}
\begin{center}
\begin{tabular}{cc}
subcase 1B & subcase 2B \\
\includegraphics[width=6cm,height=3.3cm]{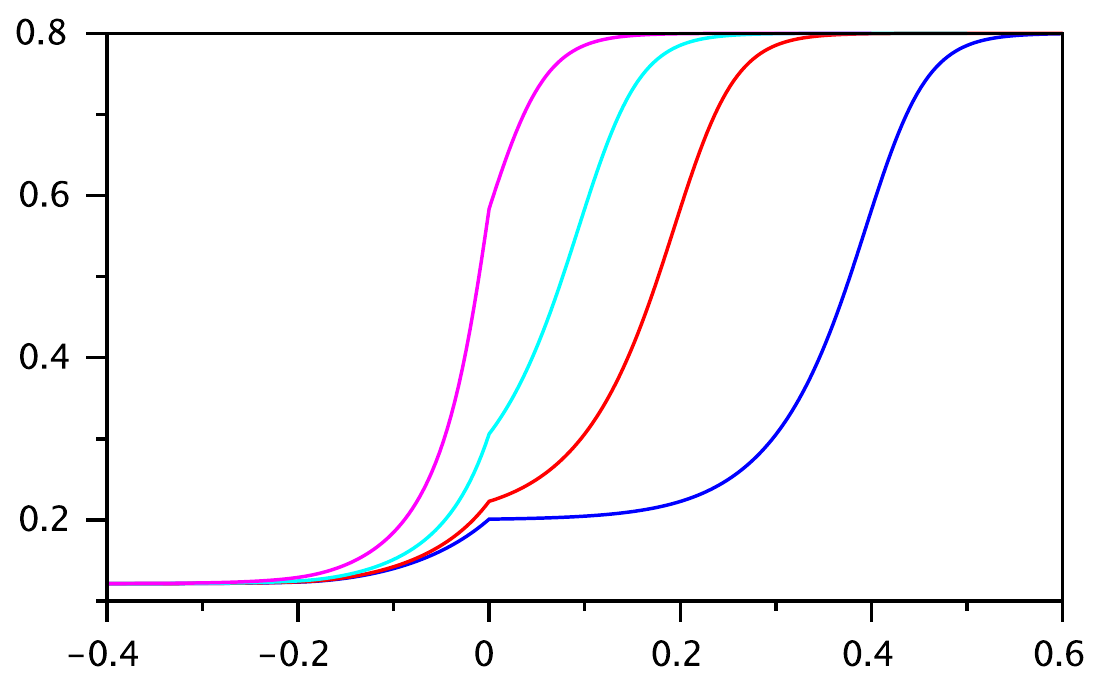}$\quad$
& 
\includegraphics[width=6cm,height=3.3cm]{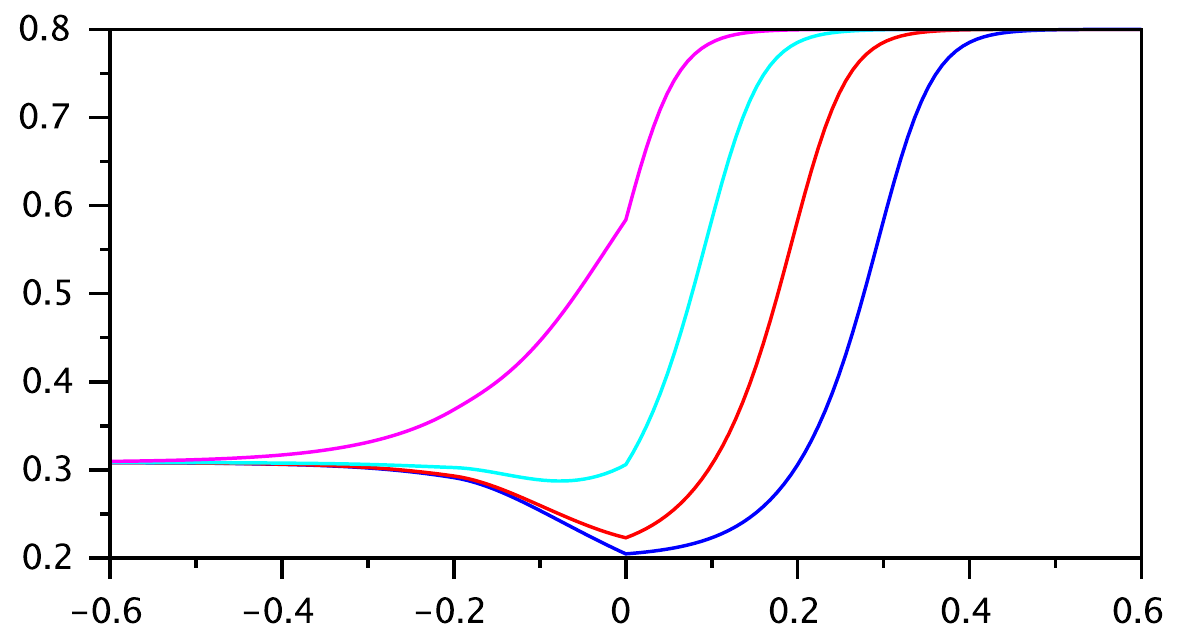}
\end{tabular}
\caption{Sample stationary wave profiles $Q$ for the nonlocal conservation law, taken from~\cite{ShenTR}.}
\label{fig:Qs}
 \end{center}
 \end{figure}

We consider subcases 1A, 1B, 2A, and 2B,  where stationary wave profiles exist. 

\begin{theorem}\label{thC-A}
Fix $h>0$ and $\bar f$. 
Let $V^\pm,\rho^\pm$ be given such that we are in one of the
subcases 1A, 1B, 2A, and 2B.  
Let $\mathcal{P}^\ell$ be a stationary wave profile for the FtLs models 
such that  $\mathcal{P}^\ell(0)=\rho_0$.
%
Then, as $\ell\to 0+$, the sequence $\{\mathcal{P}^\ell\}$ converges to a limit profile $Q$.
The profile $Q$ is a stationary profile for the conservation law~\eqref{eq:clawNL},
with $Q(0)=\rho_0$. 
\end{theorem}

\begin{proof}
On $x\ge 0$, since $V(x)\equiv V^+$ is constant, the micro-macro convergence
follows from the result in~\cite{RidderShen2018}.
%
On $x\le 0$, 
since the set of functions $\{\mathcal{P}^\ell\}$  is equicontinuous,
by the Arzel\`{a}-Ascoli Theorem, as $\ell\to 0$
the sequence $\{\mathcal{P}^\ell\}$ converges to a limit function $\hat P$ 
uniformly on bounded sets. 
Furthermore, since the asymptotic condition 
$\lim_{x\to-\infty}\mathcal{P}^\ell(x)=\rho^-$ is satisfied for all $\ell$, 
we conclude that the convergence is uniform for all $x\in\mathbb{R}$. 

It remains to show that $\hat P(x) \equiv Q(x)$ for $x\le0$, 
i.e.,  $\hat P$ satisfies the equation~\eqref{eq:Q}.
Indeed, let $P^\ell_{\{x\}}$ be the piecewise constant function generated by $\mathcal{P}^\ell$
according to~\eqref{eq:Pell},
using the periodicity~\eqref{eq:pp} we get
\[
\frac{\mathcal{P}^\ell(x)}{\bar f} 
=
\left( \frac{\ell}{\mathcal{P}^\ell(x)}\right)^{-1} 
\int_x^{x+\ell/\mathcal{P}^\ell(x)} \frac{1}{v^*(z; P^\ell_{\{z\}})}\; dz 
=
\substack{\text{\normalsize average}\\ {\scriptstyle z\in[x,x+\ell/\mathcal{P}^\ell(x)]} }
\left\{\frac{1}{v^*(z; P^\ell_{\{z\}})}\right\}.
\]
%
Taking the limit $\ell\to 0$, the above equation gives
\[
\frac{\hat P(x)}{\bar f} = \lim_{\ell\to 0} \left[\substack{\text{\normalsize average}\\ {\scriptstyle z\in[x,x+\ell/\mathcal{P}^\ell(x)]} }
\left\{\frac{1}{v^*(z;  P^\ell_{\{z\}})}\right\} \right]
= \frac{1}{ \mathcal{A}(x; V, \hat P)}
\qquad \forall x<0,
\]
which is exactly the equation~\eqref{eq:Q},
completing the proof. 
%
\end{proof}

\paragraph{Limit 2: Nonlocal to local.}
Fix $\ell>0$ and let $h\to 0$, we formally 
obtain the local particle model~\eqref{eq:FtL}. 
The stationary profile $U$ satisfies
a discontinuous delay differential equation (DDDE)
\begin{equation}\label{eq:U}
{\color{black}
U'(x) = \frac{U(x)^2}{\ell V(x) \phi(U(x))} \cdot \left[ V(x) \phi(U(x)) - V(x^\sharp) \phi(U(x^\sharp)) \right] , 
\qquad x^\sharp = x + \frac{\ell}{U(x)}.}
\end{equation}
The DDDE is studied in detail in~\cite{ShenDDDE2017}, where similar results 
on existence, uniqueness and stability are established. 
Typical profiles for case  1B and 2B are shown in Figure~\ref{fig:Us},
taken from~\cite{ShenDDDE2017}. 
We have a similar convergence result.

\begin{figure}
\begin{center}
\begin{tabular}{cc}
subcase 1B & subcase 2B \\
\includegraphics[width=6cm,height=3.5cm]{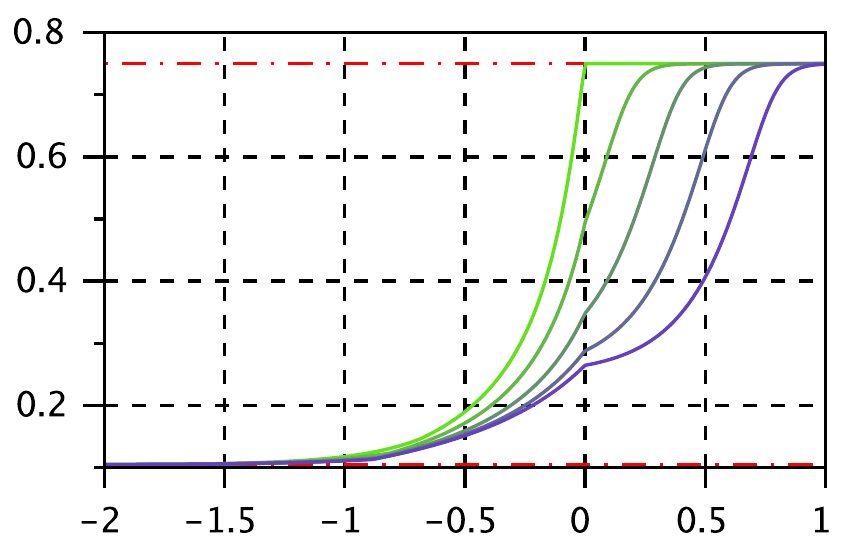} &
\includegraphics[width=6cm,height=3.5cm]{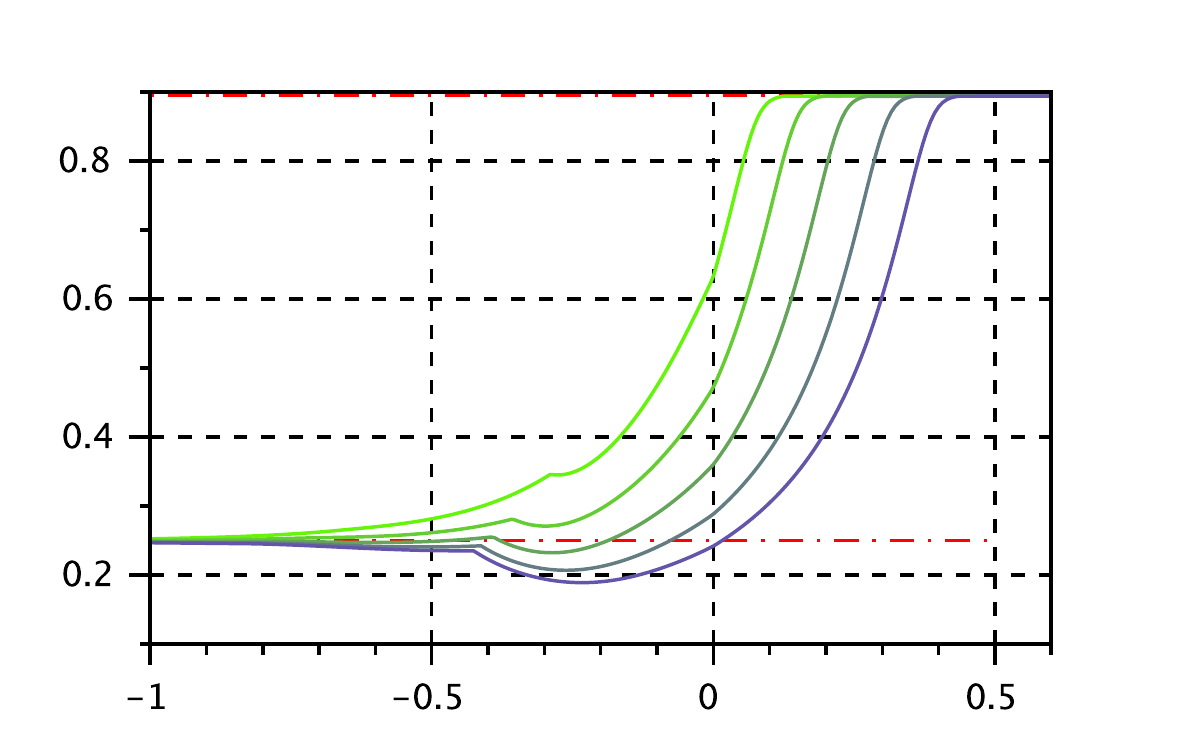}
\end{tabular}
\caption{Sample stationary wave profiles $U$ for local particle model, taken from~\cite{ShenDDDE2017}. }
\label{fig:Us}
 \end{center}
 \end{figure}

\begin{theorem}\label{thm:CB}
Fix $\ell>0, \bar f>0$ and let $h>0$ be given. 
Let $V^\pm, \rho^\pm$ satisfy the conditions
for subcases 1A, 1B, 2A, or 2B.
Denote by $\{\mathcal{P}^h\}$ the set of profiles for various $h$ such that 
$\mathcal{P}^h(0)=\tilde \rho$ for all $h$.
Then, as $h\to 0$, the sequence $\{\mathcal{P}^h\}$ converges to the profile $U$
which satisfies~\eqref{eq:U}, the asymptotic conditions
$\lim_{x\to\pm\infty}U(x)=\rho^\pm$, and 
$U(0) = \tilde\rho$.
\end{theorem}

\begin{proof}
The proof is very similar to the proof of Theorem~\ref{thC-A}.
Since the sequence $\{\mathcal{P}^h\}$ is equicontinuous, 
by the Arzel\`{a}-Ascoli Theorem
it converges to a limit function $\tilde P$ uniformly on bounded sets, as $h\to 0$.
To show that $\tilde P \equiv U$, we observe that, as $h\to 0$, 
$w\to \delta_0$ (a Dirac delta at the origin), and we have 
\[
\lim_{h\to 0} \mathcal{A}(x;V,\mathcal{P}^h) = V(x) \phi( \tilde P(x)), \qquad
\lim_{h\to 0} \mathcal{A}(L^{\mathcal{P}^h}(x);V,\mathcal{P}^h) = V(L^{\tilde P}(x)) \phi( \tilde P(L^{\tilde P}(x))),
\]
for every $x\in\mathbb{R}$. Recall that $L^{\tilde P}(x) = x+\frac{\ell}{\tilde P(x)}$. 
Thus, by~\eqref{eq:dPx} we conclude that 
$\tilde P$ satisfies the equation~\eqref{eq:U},
as well as the asymptotic conditions and the condition at $x=0$,
completing the proof. 
\end{proof}

\section{Concluding remarks}\label{sec:final}
\setcounter{equation}{0}

In this paper we study stationary wave profiles for a nonlocal particle model of
traffic flow on rough roads where the speed limit function $V$ is discontinuous at $x=0$. 
We establish results on the existence, unique, and stability for 
these profiles for all cases.

As comparison, 
we present another nonlocal particle model where
\begin{equation}\label{FtLs-2}
 \dot z_i(t)  = V(z_i) \cdot \phi(\rho^*(t,z_i)),\qquad
 \mbox{where}\quad
 \rho^*(t,z_i) \defeq  
 \int_{z_i}^{z_i+h} \rho^\ell(t,y) w(y-z_i)\, dy,
 \end{equation}
and $\rho^\ell$  is the piecewise constant function defined in~\eqref{eq:rhoell}. 
Here the weighted average is taken over the discrete density function.

When the road condition is uniform with $V(x) \equiv 1$, 
travelling wave profiles are studied in~\cite{RidderShen2018}. 
In the case of rough road condition,
similar analysis on stationary wave profiles 
can be carried out
and similar results can be proved, with small changes in the detail.
In Figure~\ref{fig:M1s} we present sample profiles for subcases 1A, 1B, 2A and 2B.


Unfortunately, model~\eqref{FtLs-2} has a fatal flaw.
In certain situations the model leads to traffic accidents, 
where the discrete density
$\rho_i(t)$ becomes bigger than 1, even with initial density less than 1.
See Figure~\ref{fig:crashes} for two numerical simulations that demonstrate
this scenario, where we use $(V^-, V^+)=(2, 1)$ and  the initial conditions
\[\mbox{left plot:} ~
\rho_i(0) =  \begin{cases} 0.9 &\mbox{if}~z_i(0) <0,\\
  0.75& \mbox{if}~z_i(0) >0,
  \end{cases}
\qquad
\mbox{right plot:} ~
\rho_i(0) =  \begin{cases} 0.9 &\mbox{if}~z_i(0) <0,\\
  0.25& \mbox{if}~z_i(0) >0.
  \end{cases}
\]
We observe that,   at the origin  $\rho_i(t)$ 
has a peak  larger than 1, for some $t>0$.

\begin{figure}
\begin{center}
\begin{tabular}{cc}
subcases 1A and 1B & subcases 2A and 2B \\
\includegraphics[width=6cm,height=3.5cm]{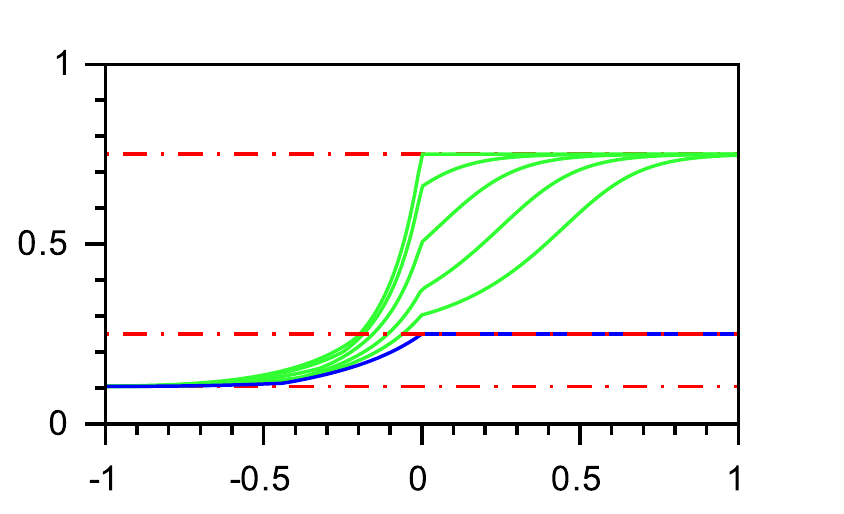}&
\includegraphics[width=6cm,height=3.5cm]{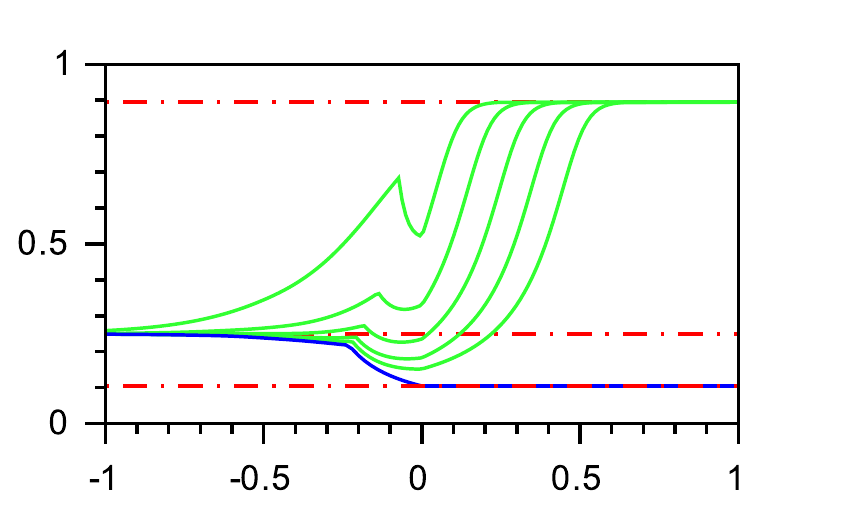}\\
\end{tabular}
\caption{Typical profiles for the alternative model, for various subcases.}
\label{fig:M1s}

\vskip 5mm 

\begin{tabular}{cc}
\includegraphics[width=5.5cm,height=3.5cm]{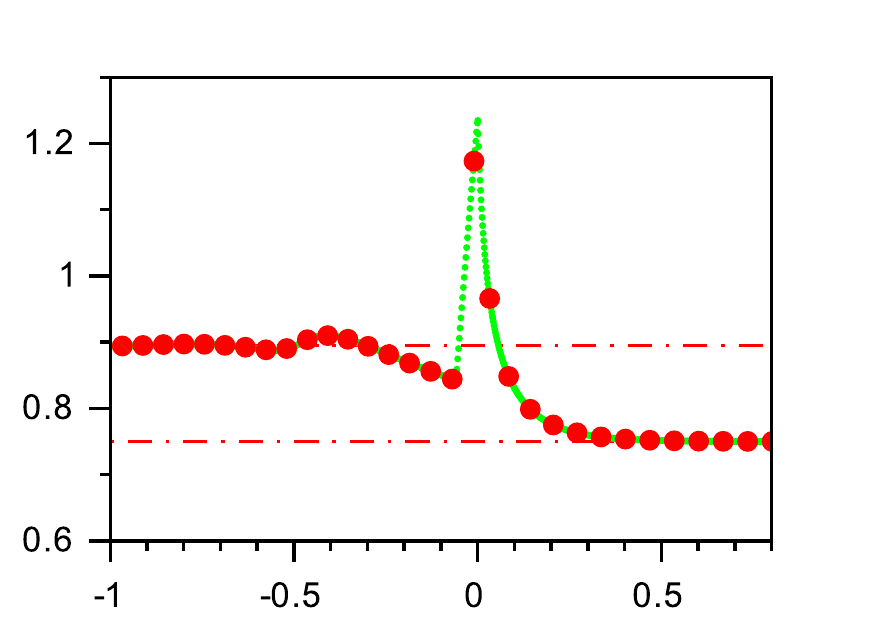}&$\quad$
\includegraphics[width=5.5cm,height=3.5cm]{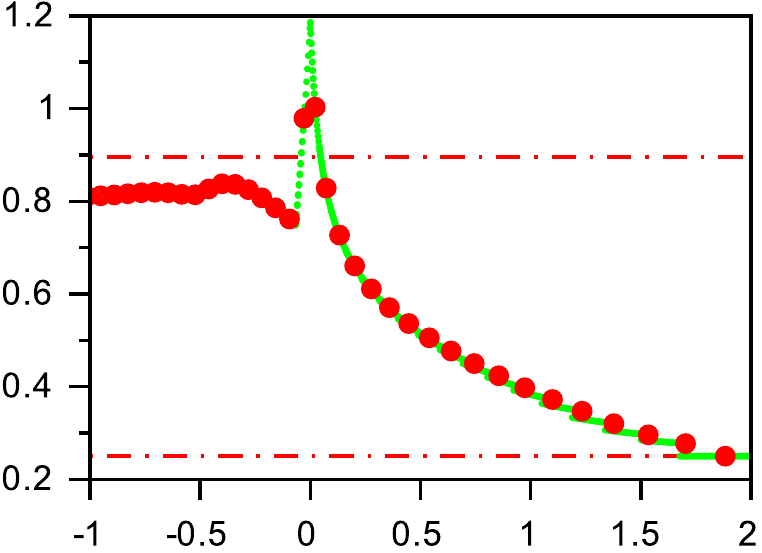}\\
\end{tabular}
\caption{Two simulations at $t=1$, showing that, when $V^->V^+$, 
car density becomes bigger than 1 as time grows,
indicating cars crashing.}
\label{fig:crashes}
 \end{center}
 \end{figure}


Numerical simulations for generating the plots used in this paper are carried out in Scilab.
The source codes can be found  at:
 
   \emph{www.personal.psu.edu/wxs27/SIM/Traffic-CS-2019} 

\medskip
\noindent\textbf{Acknowledgement.}  The authors are grateful to the anonymous reviewer for
useful remarks that led to an improvement of the manuscript. 
 


\begin{thebibliography}{99}

\bibitem{AggarwalColomboGoatin2015}
  \newblock A.~Aggarwal, R.~M.~Colombo, P.~Goatin.
  \newblock Nonlocal systems of conservation laws in several space dimensions.
  \newblock \emph{SIAM J.~Numer.~Anal.}, \textbf{53} {(2015)}, 963--983.
  

\bibitem{AggarwalGoatin2016}
  \newblock A.~Aggarwal, P.~Goatin.
  \newblock Crowd dynamics through nonlocal conservation laws.
  \newblock \emph{B.~Braz.~Math.~Soc.}, \textbf{47} {(2016)}, 37--50.

%
%
%
%
%
%
  
\bibitem{BG2016}
  \newblock S.~Blandin, P.~Goatin.
  \newblock Well-posedness of a conservation law with nonlocal flux arising in traffic flow modeling.
 \newblock \textit{Numer.~Math.} \textbf{132} (2016), 217--241.

%
%
%
%
%
%
%



\bibitem{ChenChristoforou2007}
  \newblock G.-Q.~Chen, C.~Christoforou.
  \newblock Solutions for a nonlocal conservation law with fading memory.
  \newblock \textit{Proc.~Amer.~Math.~Soc.} \textbf{135} (2007), 3905--3915. 



\bibitem{CCS2019}
\newblock M.~Colombo, G.~Crippa, L.V.~Spinolo. 
\newblock On the Singular Local Limit for Conservation Laws with Nonlocal Fluxes.
\newblock \textit{Arch.~Ration.~Mech.~Anal.} \textbf{233} (2019), no 3, 1131--1167.

\bibitem{CCS2018}
\newblock M.~Colombo, G.~Crippa, L.V.~Spinolo.
\newblock Blow-up of the total variation in the local limit of a nonlocal traffic model.
\newblock Preprint 2018, arxiv:1902.06970.

\bibitem{CCS2019P}
\newblock M.~Colombo, G.~Crippa, M.~Graff, L.V.~Spinolo.
\newblock On the role of numerical viscosity in the study of the local limit of nonlocal conservation laws.
\newblock Preprint 2019, arxiv:1902.07513.


 
\bibitem{ColomboLecureuxMercier2011}
  \newblock R.~M.~Colombo, M.~L\'{e}cureux-Mercier.
  \newblock Nonlocal crowd dynamics models for several populations.
  \newblock \emph{Acta Math.~Sci.},  \textbf{32} {(2012)}, 177--196.

\bibitem{ColomboGaravelloLecureuxMercier2011}
  \newblock R.~M.~Colombo, M.~Garavello, M.~L\'{e}cureux-Mercier.
  \newblock Nonlocal crowd dynamics.
  \newblock \emph{C.~R.~Acad.~Sci.~Paris, Ser.~I}, \textbf{349} {(2011)}, 769--772.

\bibitem{ColomboGaravelloLecureuxMercier2012}
  \newblock R.~M.~Colombo, M.~Garavello, M.~L\'{e}cureux-Mercier.
  \newblock A class of nonlocal models for pedestrian traffic.
  \newblock \emph{Math.~Models Methods Appl.~Sci.}, \textbf{22} {(2012)}.

\bibitem{ColomboMarcelliniRossi2016}
  \newblock R.~M.~Colombo, F.~Marcellini, E.~Rossi.
  \newblock Biological and industrial models motivating nonlocal conservation laws: A review of analytic and numerical results.
  \newblock \emph{Netw.~Heterog.~Media}, \textbf{11} {(2016)}, 49--67.

  
\bibitem{MR3217759} 
     \newblock    R.~M.~Colombo, E.~Rossi.
     \newblock On the micro-macro limit in traffic flow.
     \newblock \emph{Rend.~Semin.~Mat.~Univ.~Padova}, \textbf{131} (2014), 217--235.

%


 
\bibitem{CrippaLecureuxMercier2013}
  \newblock G.~Crippa, M.~L\'{e}cureux-Mercier.
  \newblock Existence and uniqueness of measure solutions for a system of continuity equations with nonlocal flow.
  \newblock \emph{NoDEA}, \textbf{20} {(2013)}, 523--537.


\bibitem{MR3541527} 
     \newblock  E.~Cristiani, S.~Sahu.
      \newblock On the micro-to-macro limit for first-order traffic flow models on
   networks,
     \newblock \emph{Netw.~Heterog.~Media}, \textbf{11} (2016), 395--413.     



\bibitem{DFFRR2017}
 \newblock M.~Di Francesco, S.~Fagioli, M.D.~Rosini, G.~Russo.
 \newblock Follow-the-Leader approximations of macroscopic models for vehicular and pedestrian flows. 
 \newblock \textit{Active particles}, Vol 1. (Bellomo, Degond, and Tadmor Eds.), 
 Birkh\"{a}user Basel, 333--378, (2017).
 


\bibitem{DFFR2019}
\newblock M.~Di Francesco, S.~Fagioli, E.~Radici.
 \newblock Deterministic particle approximation for nonlocal transport equations with nonlinear mobility.
 \newblock J.~Diff.~Eq., 266 (5), 2830--2868, (2019).

 
\bibitem{MR3356989} 
     \newblock  M.~Di Francesco, M.~D.~Rosini.
     \newblock Rigorous derivation of nonlinear scalar conservation laws from
   follow-the-leader type models via many particle limit.
     \newblock \emph{Arch.~Ration.~Mech.~Anal.}, \textbf{217} (2015), 831--871.


\bibitem{FS2019}
  \newblock M.~Di Francesco, G.~Stivaletta.
  \newblock Convergence of the follow-the-leader scheme for scalar conservation laws with space dependent flux.
  \newblock Preprint 2019.  ArXiv:1901.03618.



\bibitem{MR0477368} 
     \newblock R.~D.~Driver.
     \newblock \emph{Ordinary and delay differential equations}.
     \newblock \emph{Applied Mathematical Sciences}, \textbf{20} 
     \newblock Springer-Verlag, New York-Heidelberg, (1977).


\bibitem{MR0141863} 
     \newblock  R.~D.~Driver, M.~D.~Rosini.
     \newblock Existence and stability of solutions of a delay-differential system.
     \newblock \emph{Arch.~Ration.~Mech.~Anal.}, \textbf{10} (1962), 401--426.


\bibitem{DuKammLehoucqParks2012}
  \newblock Q.~Du, J.~R.~Kamm, R.~B.~Lehoucq, M.~L.~Parks.
    \newblock A new approach for a nonlocal, nonlinear conservation law.
  \newblock \emph{SIAM J.~Appl.~Math.}, \textbf{72} {(2012)}, 464--487.

%


\bibitem{FKG2018} 
\newblock    J.~Friedrich, O.~Kolb, S.~G\"{o}ttlich.
\newblock A Godunov type scheme for a class of LWR traffic flow models with non-local flux, \newblock \emph{Netw.~Heterog.~Media}, Vol. \textbf{13} (2018), pp. 531--547.


\bibitem{MR3605557} 
     \newblock  P.~Goatin, F.~Rossi.
      \newblock A traffic flow model with non-smooth metric interaction:
   well-posedness and micro-macro limit.
     \newblock \emph{Commun.~Math.~Sci.}, \textbf{15} (2017), 261--287.


%
%


\bibitem{HoldenRisebro} 
     \newblock  H.~Holden, N.~H.~Risebro.
     \newblock Continuum limit of Follow-the-Leader models -- a short proof,
     \newblock \emph{Discrete Contin.~Dyn.~Syst.} \textbf{38} (2018), no.~2, 715--722.

\bibitem{HoldenRisebro2}
  \newblock  H.~Holden,  N.~H.~Risebro. 
  \newblock Follow-the-Leader models can be viewed as a numerical approximation to the Lighthill-Whitham-Richards model for traffic flow. 
  \newblock \emph{Netw.~Heterog.~Media} \textbf{13} (2018), no.~3, 409--421.

%
%

\bibitem{Martin} 
\newblock R.~H.~Martin, Jr..
\newblock \textit{Nonlinear Operators and Differential Equations in Banach Spaces}.
\newblock Pure and Applied Mathematics, Wiley, New York, 1976.


%
%
     
 \bibitem{RidderShen2018}
  \newblock J.~Ridder, W.~Shen. 
  \newblock Traveling Waves for Nonlocal Models of Traffic Flow.
 \newblock  Accepted for publication in \textit{Discrete Contin.~Dyn.~Syst.--A)} 2019.   
 \newblock arXiv:1808.03734.
 
\bibitem{SellYou}	
\newblock G.~Sell, Y.~You.
\newblock \textit{Dynamics of Evolutionary Equations}.
\newblock Applied Mathematical Sciences, \textbf{143}, Springer-Verlag, New York, 2002.

%
	
\bibitem{ShenDDDE2017}
\newblock W.~Shen.
\newblock Traveling wave profiles for a Follow-the-Leader model for traffic flow  with rough road condition, 
\newblock  \textit{Netw. Heterog. Media},  \textbf{13}(2018), no.~3,  449--478.  


\bibitem{ShenTR}
\newblock W.~Shen.
\newblock Traveling Waves for Conservation Laws with Nonlocal Flux for Traffic Flow on 
Rough Roads,
\newblock Accepted for publication in \textit{Netw. Heterog. Media}, 2019.
\newblock ArXiv:1809.02998.


\bibitem{ShenKarim2017}
\newblock {W.~Shen and K.~Shikh-Khalil}.
 \newblock  {Traveling Waves for a Microscopic Model of Traffic Flow},
 \newblock  \textit{Discrete~Cont.~Dyn.~Syst.A}, \textbf{38} (2018), 2571--2589.
 
 
%

\bibitem{Whitham}
\newblock B.~Whitham.
\newblock \textit{Linear and nonlinear waves}.
\newblock Wiley \& Sons, New York, 1974. 



\bibitem{Zumbrun1999}
  \newblock K.~Zumbrun.
  \newblock On a nonlocal dispersive equation modeling particle suspensions,
  \newblock \emph{Q.~Appl.~Math.}, \textbf{57} {(1999)}, 573--600.

\end{thebibliography}
\end{document}